\documentclass{amsart}

\usepackage{amssymb}
\usepackage[all,cmtip]{xy}
\usepackage{tikz-cd}
\usetikzlibrary{matrix}
\usepackage{hyperref}

\newcommand{\isomto}{\overset{\sim}{\to}}
\newcommand{\longisomto}{\overset{\sim}{\longrightarrow}}

\newcommand{\surjto}{\twoheadrightarrow}

\newcommand{\injto}{\hookrightarrow}

\newcommand{\longto}{\longrightarrow}

\makeatletter
\newcommand*\rel@kern[1]{\kern#1\dimexpr\macc@kerna}
\newcommand*\widebar[1]{%
  \begingroup
  \def\mathaccent##1##2{%
    \rel@kern{0.8}%
    \overline{\rel@kern{-0.8}\macc@nucleus\rel@kern{0.2}}%
    \rel@kern{-0.2}%
  }%
  \macc@depth\@ne
  \let\math@bgroup\@empty \let\math@egroup\macc@set@skewchar
  \mathsurround\z@ \frozen@everymath{\mathgroup\macc@group\relax}%
  \macc@set@skewchar\relax
  \let\mathaccentV\macc@nested@a
  \macc@nested@a\relax111{#1}%
  \endgroup
}
\makeatother

\DeclareMathOperator\Sym{Sym}
\DeclareMathOperator\Aut{Aut}
\DeclareMathOperator\Out{Out}

\DeclareMathOperator\End{End}
\DeclareMathOperator\Spec{Spec}

\DeclareMathOperator\im{im}

\DeclareMathOperator\diag{diag}

\DeclareMathOperator\Pic{Pic}

\DeclareMathOperator\GL{GL}
\DeclareMathOperator\SL{SL}
\DeclareMathOperator\fsl{\mathfrak{sl}}
\DeclareMathOperator\fgl{\mathfrak{gl}}
\DeclareMathOperator\fso{\mathfrak{so}}
\DeclareMathOperator\ad{ad}
\DeclareMathOperator\SO{SO}
\DeclareMathOperator\bSO{\mathbf{SO}}
\DeclareMathOperator\bSL{\mathbf{SL}}
\DeclareMathOperator\rO{O}
\DeclareMathOperator\Spin{Spin}
\DeclareMathOperator\bSpin{\mathbf{Spin}}
\DeclareMathOperator\GSpin{GSpin}
\DeclareMathOperator\bGSpin{\mathbf{GSpin}}

\DeclareMathOperator\DMon{DMon}
\DeclareMathOperator\Mon{Mon}
\DeclareMathOperator\Gal{Gal}

\DeclareMathOperator\Ext{Ext}

\DeclareMathOperator\ch{ch}

\def\bQ{{\mathbf{Q}}} \def\bZ{{\mathbf{Z}}} 
 \def\bG{{\mathbf{G}}} \def\bR{{\mathbf{R}}}
\def\bC{{\mathbf{C}}}

\def\cO{{\mathcal{O}}}  
\def\cF{{\mathcal{F}}} \def\cE{{\mathcal{E}}} \def\cL{{\mathcal{L}}} 
 \def\cB{{\mathcal{B}}}
  
 \def\cD{{\mathcal D}} \def\cQ{{\mathcal{Q}}}
  \def\cX{{\mathcal{X}}}
 \def\cP{{\mathcal P}}

  \def\rH{{\rm H}} \def\rmd{{\rm d}}
  \def\rHH{{\rm HH}} \def\rK{{\mathrm K}}
\def\SH{{\mathrm{SH}}}

\def\fS{{\mathfrak{S}}} \def\fg{{\mathfrak{g}}} \def\fa{{\mathfrak{a}}}

\def\pt{{\rm pt}}

\def\an{{\rm an}}

\def\ev{{\rm ev}}
\def\odd{{\rm odd}}

\def\top{{\rm top}}

\def\Td{{\rm Td}}

\def\BKR{{\rm BKR}}
\def\NS{{\rm NS}}

\theoremstyle{plain}
\newtheorem{theorem}{Theorem}[section]
\newtheorem*{theorem*}{Theorem}
\newtheorem{corollary}[theorem]{Corollary}
\newtheorem*{corollary*}{Corollary}
\newtheorem{lemma}[theorem]{Lemma}
\newtheorem{bigthm}{Theorem}
\newtheorem{proposition}[theorem]{Proposition}
\theoremstyle{definition}
\newtheorem{definition}[theorem]{Definition}
\newtheorem*{definition*}{Definition}
\newtheorem*{remark*}{Remark}

\newtheorem{remark}[theorem]{Remark}

\newtheorem*{question*}{Question}

\begin{document}

\title
 {Derived equivalences of hyperk\"ahler varieties}

\author
 {Lenny Taelman}

\begin{abstract}
We show that the Looijenga--Lunts--Verbitsky Lie algebra acting on the cohomology of a hyperk\"ahler variety is a derived invariant, and obtain from this a number of consequences for the action on cohomology of derived equivalences between hyperk\"ahler varieties. 

This includes a proof that derived equivalent hyperk\"ahler varieties have isomorphic $\bQ$-Hodge structures, the construction of a rational `Mukai lattice' functorial for derived equivalences, and the computation (up to index $2$) of the image of the group of auto-equivalences on the cohomology of certain Hilbert squares of K3 surfaces. 

%
%. We use this to show that derived equivalent hyperk\"ahler varieties have isomorphic $\bQ$-Hodge structures. We also show that derived equivalences between hyperk\"ahler varieties induce isometries between rational Mukai lattices. For many K3 surfaces $S$, we compute (up to index $2$) the image of the group of auto-equivalences of the Hilbert square of $S$ on its  cohomology. 
\end{abstract}

\maketitle

\section{Introduction}

\subsection{Background}We briefly recall the background to our results. We refer to \cite{HuybrechtsFM} for more details. 
For a smooth projective complex variety $X$ we denote by $\cD X$ the bounded derived category of coherent sheaves on $X$. By a theorem of Orlov \cite{Orlov03} any (exact, $\bC$-linear) equivalence $\Phi \colon \cD X_1 \isomto \cD X_2$ comes from a Fourier--Mukai kernel $\cP \in \cD(X_1\times X_2)$, and convolution with the Mukai vector $v(\cP) \in \rH(X_1\times X_2,\bQ)$ defines an isomorphism 
\[
	\Phi^\rH\colon \rH(X_1,\bQ) \isomto \rH(X_2,\bQ)
\]
between the total cohomology of $X_1$ and $X_2$. This isomorphism is not graded, and respects the Hodge structures only up to Tate twists. Nonetheless, Orlov has conjectured \cite{Orlov05} that if $X_1$ and $X_2$ are derived equivalent, then for every $i$ there exist (non-canonical) isomorphisms $\rH^i(X_1,\bQ)\cong \rH^i(X_2,\bQ)$ of $\bQ$-Hodge structures.
% We show that this is indeed the case if $X_1$ and $X_2$ are hyperkh\"ahler varieties.

For every $X$ we have a representation
\[
	\rho_X\colon  \Aut(\cD X) \to \GL(\rH(X,\bQ)),\, \Phi \mapsto \Phi^\rH.
\]
Its image is known for varieties with ample or anti-ample canonical class (in which case $\Aut(\cD X)$ is small and well-understood \cite{BondalOrlov01}), for abelian varieties \cite{GolyshevLuntsOrlov01}, and for K3 surfaces.
% In this paper, we study the image of $\rho_X$ for higher-dimensional hyperk\"ahler varieties. 
To place our results in context, we recall the description of the image for K3 surfaces. 

Let $X$ be a K3 surface. Consider the \emph{Mukai lattice}
\[
	\widetilde\rH(X,\bZ) := \rH^0(X,\bZ) \oplus \rH^2(X,\bZ(1)) \oplus \rH^4(X,\bZ(2)).
\]
This is a Hodge structure of weight $0$, and it comes equipped with a perfect bilinear form $b$ of signature $(4,20)$. For convenience, we denote by $\alpha$ and $\beta$   the natural generators of $\rH^0(X,\bZ)$ and $\rH^4(X,\bZ(2))$ respectively, so that we have
$\widetilde\rH(X,\bZ) = \bZ \alpha \oplus \rH^2(X,\bZ(1)) \oplus \bZ \beta$. The pairing $b$ is the orthogonal sum of 
 the intersection pairing on $\rH^2(X,\bZ(1))$ and the pairing on $\bZ\alpha \oplus \bZ\beta$ given by $b(\alpha,\alpha)=b(\beta,\beta)=0$ and $b(\alpha,\beta)=-1$.

It was observed by Mukai \cite{Mukai87} that if $\Phi\colon \cD X_1 \isomto \cD X_2$ is a derived equivalence between K3 surfaces, then $\Phi^\rH$ restricts to an isomorphism $\Phi^{\widetilde\rH}\colon \widetilde\rH(X_1,\bZ) \to \widetilde\rH(X_2,\bZ)$ respecting the pairing and Hodge structures. Denote by $\Aut(\widetilde\rH(X,\bZ))$ the group of isometries of $\widetilde\rH(X,\bZ)$ respecting the Hodge structure, and  by $\Aut^+(\widetilde\rH(X,\bZ))$ the subgroup (of index $2$) consisting of those isometries that respect the orientation on a four-dimensional positive definite subspace of $\widetilde\rH(X,\bR)$.

\begin{theorem}[\cite{Mukai87,Orlov97,HLOY,Ploog05,HMS}]
Let $X$ be a K3 surface. Then the image of  $\rho_X$ is $\Aut^+(\widetilde \rH(X,\bZ))$.\qed
\end{theorem}

In this paper, we prove Orlov's conjecture on $\bQ$-Hodge structures for hyperk\"ahler varieties, construct a rational version of the Mukai lattice for hyperk\"ahler  varieties, and compute (up to index $2$) the image of $\rho_X$ for certain Hilbert squares of K3 surfaces. The main tool in these results is the Looijenga--Lunts--Verbitsky Lie algebra.

\subsection{The LLV Lie algebra and derived equivalences}

Let $X$ be a smooth projective complex variety. By the Hard Lefschetz theorem, every ample class $\lambda \in \NS(X)$ determines a Lie algebra $\fg_\lambda \subset \End(\rH(X,\bQ))$ isomorphic to $\fsl_2$. More generally, this holds for every cohomology class $\lambda \in \rH^2(X,\bQ)$ (algebraic or not) satisfying the conclusion of the Hard Lefschetz theorem. Looijenga and Lunts \cite{LooijengaLunts97}, and Verbitsky \cite{Verbitsky96} have studied the Lie algebra $\fg(X) \subset \End(\rH(X,\bQ))$ generated by the collection of the Lie algebras $\fg_\lambda$. We will refer to this as the LLV Lie algebra. See \S~\ref{sec:Looijenga--Lunts} for more details. 

We say that $X$ is \emph{holomorphic symplectic} if it admits a nowhere degenerate holomorphic symplectic form $\sigma \in \rH^0(X,\Omega^2_X)$.

\begin{bigthm}[\S~\ref{subsec:Lie-invariant}]\label{bigthm:Lie-invariant}Let $X_1$ and $X_2$ be holomorphic symplectic varieties. Then for every equivalence $\Phi\colon \cD X_1 \isomto \cD X_2$ there exists a canonical isomorphism of rational Lie algebras
\[
	\Phi^\fg\colon \fg(X_1) \isomto \fg(X_2)
\]
with the property that the map $\Phi^\rH \colon \rH(X_1,\bQ) \isomto \rH(X_2,\bQ)$ is equivariant with respect to $\Phi^\fg$.
 \end{bigthm}
  
Note that $\fg(X)$ is defined in terms of the grading and the cup product on $\rH(X,\bQ)$, neither of which are preserved under derived equivalences. 

To prove Theorem~\ref{bigthm:Lie-invariant} we introduce a complex Lie algebra $\fg'(X)$  whose definition is similar to the rational Lie algebra $\fg(X)$, but where the action of $\rH^2(X,\bQ)$ on $\rH(X,\bQ)$ is replaced with a natural action of the Hochschild cohomology group $\rHH^2(X)$ on Hochschild homology $\rHH_\bullet(X)$. Since Hochschild cohomology and its action on Hochschild homology is known to be invariant under derived equivalences, it follows that $\fg'(X)$ is a derived invariant. We  show that if $X$ is holomorphic symplectic, then the isomorphism $\rHH_\bullet(X) \to \rH(X,\bC)$ (coming from the Hochschild--Kostant--Rosenberg isomorphism) maps $\fg'(X)$ to $\fg(X)\otimes_\bQ \bC$. This is closely related to Verbitsky's `mirror symmetry' for hyperk\"ahler varieties \cite{Verbitsky96,Verbitsky99}. From this we  deduce that
the rational Lie algebra $\fg(X)$ is a derived invariant.

\subsection{A rational Mukai lattice for hyperk\"ahler varieties}
  
Let $X$ be a hyperk\"ahler variety.  Consider the $\bQ$-vector space
\[
 	\widetilde\rH(X,\bQ) := \bQ\alpha \oplus \rH^2(X,\bQ) \oplus \bQ\beta
\]
equipped with the bilinear form $b$ which is the orthogonal sum of the Beauville--Bogomolov form on $\rH^2(X,\bQ)$ and a hyperbolic plane $\bQ\alpha \oplus \bQ\beta$ with $\alpha$ and $\beta$ isotropic and $b(\alpha,\beta)=-1$. By analogy with the case of a K3 surface, we will call $\widetilde\rH(X,\bQ)$ the (rational) \emph{Mukai lattice} of $X$. Looijenga--Lunts \cite{LooijengaLunts97} and Verbitsky \cite{Verbitsky96} have shown that the Lie algebra $\fg(X)$ can be canonically identified with $\fso(\widetilde\rH(X,\bQ))$, see \S~\ref{sec:LL-so} for a precise statement.
Moreover, Verbitsky \cite{Verbitsky96} has shown that the sub-algebra $\SH(X,\bQ)$ of $\rH(X,\bQ)$ generated by $\rH^2(X,\bQ)$ forms an irreducible sub-$\fg(X)$-module.  Using this, we show that Theorem~\ref{bigthm:Lie-invariant} implies:

\begin{bigthm}[\S~\ref{subsec:verbitsky-component}]\label{bigthm:hyperkahler-subring}
Let $X_1$ and $X_2$ be hyperk\"ahler varieties and $\Phi\colon \cD X_1 \isomto \cD X_2$ an equivalence. Then the induced isomorphism $\Phi^\rH$ restricts  to an isomorphism
$\Phi^{\SH}\colon \SH(X_1,\bQ) \isomto \SH(X_2,\bQ).$
\end{bigthm}

%The $\fg(X)$-module $\SH(X,\bQ)$ can be canonically described in terms of the rational Mukai lattice $\widetilde\rH(X,\bQ)$. Using this, we show that a Fourier--Mukai equivalence $\cD X_1 \isomto \cD X_2$,
%induces a Hodge isometry
%\[
%	\Phi^{\widetilde\rH} \colon \widetilde\rH(X_1,\bQ) \isomto \widetilde\rH(X_2,\bQ).
%\]
%(If the dimension of the $X_i$ are divisible by $4$, then the construction of 

%Using Theorem~\ref{bigthm:Lie-invariant} and Theorem~\ref{bigthm:hyperkahler-subring} we show:
%
%\begin{bigthm}[\S~\ref{subsec:verbitsky-component}]\label{bigthm:induced-mukai-map-odd}
%Let $X_1$ and $X_2$ be hyperk\"ahler varieties of dimension $2d$ with $d$ odd.  Let $\Phi\colon \cD X_1 \isomto \cD X_2$ a Fourier--Mukai equivalence. Then there exists a unique Hodge isometry
%inducing $\Phi^\SH$.
%\end{bigthm}
%
%
%If $d$ is even and $b_2(X_i)$ is odd, then there is a similar, but slightly more complicated, statement for \emph{oriented} hyperk\"ahler varieties. See \S~\ref{subsec:verbitsky-component} for more details. 
%
Taking $X_1=X_2=X$ in Theorem~\ref{bigthm:hyperkahler-subring} we obtain a homomorphism  
\[
	\rho_{X}^{\SH} \colon \Aut(\cD X) \longto \GL (\SH(X,\bQ)).
\]
The complex structure on a hyperk\"ahler variety $X$ induces a Hodge structure of weight~$0$ on $\widetilde\rH(X,\bQ)$ given by
\[
	\widetilde\rH(X,\bQ) = \bQ\alpha \oplus \rH^2(X,\bQ(1)) \oplus \bQ\beta.
\]
Denote by $\Aut \widetilde\rH(X,\bQ)$ the group of Hodge isometries of $\widetilde\rH(X,\bQ)$. 

\begin{bigthm}[\S~\ref{subsec:verbitsky-component}]\label{bigthm:normalizer}
Let $X$ be a hyperk\"ahler variety of dimension $2d$ and second Betti number $b_2$. Assume that 
$b_2$ is odd or $d$ is odd. Then $\rho_X^\SH$ factors over a map $\rho_X^{\widetilde\rH}\colon \Aut(\cD(X)) \to \Aut(\widetilde\rH(X,\bQ))$.
\end{bigthm}

See \S~\ref{subsec:regular-part} and \S~\ref{subsec:verbitsky-component} for an explicit description of the implicit map
$\Aut(\widetilde\rH(X,\bQ)) \to \GL(\SH(X,\bQ))$. 

Note that all known  hyperk\"ahler varieties satisfy the parity conditions in the theorem: there are two infinite series of deformation classes with odd $b_2$ (generalized Kummers and Hilbert schemes of points), and three exceptional deformation classes with odd $d$ (K3, OG6, OG10) .

%We also show a version of Theorem~\ref{bigthm:normalizer} for derived equivalences between distinct hyperk\"ahler varieties. If $X_1$ and $X_2$ are hyperk\"ahler varieties of dimension $2d$ with $d$ odd, then we associate to any equivalence $\Phi\colon \cD X_1 \isomto \cD X_2$ a canonical Hodge isometry
%\[
%	\Phi^{\widetilde\rH} \colon \widetilde \rH(X_1,\bQ) \to \widetilde\rH(X_2,\bQ).
%\]
%If $d$ is even, and $b_2$ is odd, then  a similar result holds for `oriented' hyperk\"ahler varieties,  see \S~\ref{subsec:verbitsky-component}.

%Taking into account that derived equivalences preserve topological $K$-theory (see \cite{AddingtonThomas14}), one concludes that there exists an arithmetic subgroup $\Gamma \subset \rO(\widetilde\rH(X,\bQ))$ containing the image of $\rho^A_X$. 
%%We will use this strategy to study the image of $\rho_X$ for hyperk\"ahler varieties $X$ of type ${\mathrm{K3}}^{[2]}$.
%

\subsection{Hodge structures of derived equivalent hyperk\"ahler varieties}

Another application of Theorem~\ref{bigthm:Lie-invariant} is the following.

\begin{bigthm}[\S~\ref{sec:hodge}]\label{bigthm:hodge-structures} Let $X_1$ and $X_2$ be derived equivalent hyperk\"ahler varieties. Then for every $i$ the $\bQ$-Hodge structures $\rH^i(X_1,\bQ)$ and $\rH^i(X_2,\bQ)$ are isomorphic.
\end{bigthm}

This confirms Orlov's conjecture for hyperk\"ahler varieties. The proof is inspired by \cite{Soldatenkov}.

\subsection{Auto-equivalences of the Hilbert square of a K3 surface}
In the second half of the paper we consider the problem of determining the image of $\rho_X$ for certain hyperk\"ahler varieties. An important difference with the first half of the paper is that \emph{integral} structures (lattices, arithmetic subgroups, \ldots) will play an important role here.

As a first approximation to determining the image of $\rho_X$, we consider a variation of this problem which is deformation invariant.
Let $X$ be a smooth projective complex variety. If $X'$ and $X''$ are smooth deformations of  $X$ (parametrized by paths in the base), and if $\Phi \colon \cD X' \isomto \cD X''$ is an equivalence, then we obtain an isomorphism as the composition
\[
	\rH(X,\bQ) \longto \rH(X',\bQ) \overset{\Phi^\rH}{\longto} \rH(X'',\bQ) \longto \rH(X,\bQ).
\]
We define the \emph{derived monodromy group} of $X$ to be the subgroup $\DMon(X)$ of $\GL(\rH(X,\bQ))$ generated by all these isomorphisms. This group contains both the usual monodromy group of $X$ and the image of $\rho_X\colon \Aut(\cD X)\to\GL(\rH(X,\bQ))$.  
%It preserves the image of topological $K$-theory in $\rH(X,\bQ)$ and hence provides an `integral' upper bound for the image of $\rho_X$.

If $S$ is a K3 surface, then the result of \cite{HMS} implies $\DMon(S)=\rO^+(\widetilde\rH(S,\bZ))$, and that the image of $\rho_S$ consists of those elements of $\DMon(S)$ that respect the Hodge structure on $\widetilde\rH(S,\bZ)$. Similarly, for an abelian variety $A$, the results of \cite{GolyshevLuntsOrlov01} imply  $\DMon(A)=\Spin(\rH^{1}(A,\bZ)\oplus \rH^{1}(A^\vee,\bZ))$, and that the image of $\rho_A$ consists of those elements of $\DMon(A)$ that respect the Hodge structure on $\rH^{1}(A,\bZ)\oplus \rH^{1}(A^\vee,\bZ)$.

Now let $X$ be a hyperk\"ahler variety of type ${\mathrm{K3}}^{[2]}$. We have $\rH(X,\bQ)=\SH(X,\bQ)$ and hence by Theorem~\ref{bigthm:normalizer} the action of $\Aut(\cD X)$ on $\rH(X,\bQ)$ factors over a subgroup $\rO(\widetilde\rH(X,\bQ))$
of $\GL(\rH(X,\bQ))$.

\begin{bigthm}[\S~\ref{subsec:lower-bound}]\label{bigthm:derived-monodromy}Let $X$ be a hyperk\"ahler variety deformation equivalent to the Hilbert square of a K3 surface. 
There is an integral lattice $\Lambda \subset \widetilde\rH(X,\bQ)$ such that
\[
 	\rO^+(\Lambda) \subset \DMon(X) \subset \rO(\Lambda)
\]
inside $\rO(\widetilde\rH(X,\bQ))$.
\end{bigthm}

See \S~\ref{subsec:lower-bound} for a precise description of $\Lambda$.  As an abstract lattice, $\Lambda$ is isomorphic to $\rH^2(X,\bZ) \oplus U$, but its image in  $\widetilde\rH(X,\bQ)$ is \emph{not} $\bZ\alpha \oplus \rH^2(X,\bZ) \oplus \bZ \beta$.

Crucial in the proof of Theorem~\ref{bigthm:derived-monodromy} is the \emph{derived McKay correspondence} due to Bridgeland, King, Reid \cite{BKR} and Haiman \cite{Haiman}. It provides an ample supply of elements of $\DMon(X)$: every deformation of $X$ to the Hilbert square $S^{[2]}$ of a K3 surface $S$ induces an inclusion $\DMon(S)\to \DMon(X)$. As part of the proof, we explicitly compute this inclusion.
  
%Another important ingredient is \emph{topological $K$-theory}, which provides an integral lattice acted upon by the derived monodromy group. 
%In fact, as a by-product of the proof of Theorem \ref{bigthm:derived-monodromy}, we compute the topological $K$-theory of $X$.
%The upper bound in Theorem \ref{bigthm:derived-monodromy} comes from combining two restrictions on the image. By Theorem \ref{bigthm:Lie-invariant}, the group $\DMon(X)$ must normalize the Lie algebra $\fg(X)$. And as was observed by 

We denote by $\Aut(\Lambda)$ the group of isometries of $\Lambda \subset \widetilde\rH(X,\bQ)$ that respect the Hodge structure on $\widetilde\rH(X,\bQ)$. It follows from Theorem \ref{bigthm:derived-monodromy} that $\im(\rho_X)$ is contained in $\Aut(\Lambda)$ for every $X$ which is deformation equivalent to  the Hilbert square of a K3 surface.
For some $X$ we can show that the upper bound in the above corollary is close to being sharp. Denote by $\Aut^+(\Lambda) \subset \Aut(\Lambda)$ the subgroup consisting of those Hodge isometries that respect the orientation of a positive $4$-plane in $\Lambda_\bR$.

\begin{bigthm}[\S~\ref{subsec:the-image}]\label{bigthm:lower-bound}
Let $S$ be a complex K3 surface and $X=S^{[2]}$. Assume that $\NS(X)$ contains a hyperbolic plane. Then we have $\Aut^+(\Lambda) \subset \im(\rho_X) \subset \Aut(\Lambda)$.
\end{bigthm}

\begin{remark}To determine $\im \rho_X$ up to index $2$ for a general hyperk\"ahler of type $\mathrm{K3}^{[2]}$ new constructions of derived equivalences will be needed.
\end{remark}

\begin{remark}
Theorem~\ref{bigthm:derived-monodromy} and Theorem~\ref{bigthm:lower-bound} leave an ambiguity of index $2$, related to orientations on a maximal positive subspace of $\widetilde\rH(X,\bR)$. In the case of K3 surfaces, it was conjectured by Szendr\H{o}i \cite{Szendroi} that derived equivalences must respect such orientation, and this was proven by Huybrechts, Macr\`i, and Stellari \cite{HMS}. Their method is based on deformation to generic (formal or analytic) K3 surfaces of Picard rank~$0$, and on a complete understanding of the space of stability conditions on those \cite{HMS08}. It is far from clear if such a strategy can be used to remove the index~$2$ ambiguity for hyperk\"ahler varieties of type $\mathrm{K3}^{[2]}$.
\end{remark}

\begin{remark}
That a lattice of signature $(4,b_2-2)$ should play a role in describing the image of $\rho_X$ for hyperk\"ahler varieties $X$ was expected from the physics literature \cite{Dijkgraaf99}, but it is not clear where the lattice should come from, nor what its precise description should be for general hyperk\"ahler varieties. In the above results, the lattice $\Lambda$ arises in a rather implicit way, and one may hope for a more concrete interpretation of its elements.
\end{remark}

\begin{remark}
It is tempting to try to conjecture a description of the  group $\Aut(\cD X)$ in terms of an action on a space of stability conditions on~$X$, generalising Bridgeland's work on K3 surfaces \cite{Bridgeland}. However, there is  a representation-theoretic obstruction against doing this naively.  The central charge of a hypothetical stability condition on $X$ takes values in $\rH(X,\bC)$, yet Theorems~\ref{bigthm:derived-monodromy} and \ref{bigthm:lower-bound} suggest the central charge should take values in $\widetilde\rH(X,\bC)$. If $X$ is of type $\mathrm{K3}^{[2]}$, then $\rH(X,\bC)$ and $\widetilde\rH(X,\bC)$ are non-isomorphic irreducible $\DMon(X)$-modules, so that this would require a modification of the notion of stability condition. 
\end{remark}

%\subsection{Generalizations, and questions}
%
%$\fg(X)$ a derived invariant for non-symplectic varieties? New proof for the classical results on Abelian varieties? Multiplicities in decomposition of cohomology of hyperk\"ahler (if multiplicity-free, then good upper bound on action of $\Aut\cD X$ on total cohomology)?  Topological $K$-theory of higher-dimensional Hilbert schemes as $\DMon$-module? Kummer varieties? Bridgeland--style conjectural description of $\Aut \cD(X)$? 

\subsection{Acknowledgements}
%Verbitsky's body of work on cohomology and mirror symmetry of hyperk\"ahlers. Addington. 
I am grateful to Nick Addington, Thorsten Beckmann, Eyal Markman, and Zo\"e Schroot for many valuable comments on earlier drafts of this paper.

\section{The LLV Lie algebra of a smooth projective variety}

In this section we recall the construction of Looijenga and Lunts \cite{LooijengaLunts97} and Verbitsky \cite{Verbitsky96} of a Lie algebra acting naturally on the cohomology of algebraic varieties. For holomorphic sympletic varieties we show that this Lie algebra is a derived invariant.

\subsection{The LLV Lie algebra}
\label{sec:Looijenga--Lunts}

Let $F$ be a field of characteristic zero and let $M$ be a $\bZ$-graded $F$-vector space, of finite $F$-dimension. Denote by $h$ the endomorphism of $M$ that is multiplication by $n$ on $M_n$. 

 Let $e$ be an endomorphism of $M$ of degree $2$. We say that $e$ \emph{has the hard Lefschetz property} if for every $n\geq 0$ the map $e^n \colon M_{-n} \to M_n$ is an isomorphism. This is equivalent to the existence of an $f\in \End(M)$ such that the relations
 \begin{equation}\label{eq:sl2-triple}
 	[h,e]=2e, \quad [h,f]=-2f, \quad [e,f]=h
\end{equation}
hold in $\End(M)$. Thus, $(e,h,f)$ forms an $\fsl_2$-tripe and defines a Lie homomorphism $\fsl_2 \to \End(M)$.

\begin{proposition}\label{prop:JM} Assume that $e$ has the hard Lefschetz property. Then
the element $f$ satisfying (\ref{eq:sl2-triple}) is unique, and if $e$ and $h$ lie in a semi-simple sub-Lie algebra $\fg \subset \End(M)$, then so does $f$.
\end{proposition}

\begin{proof}[Proof of Proposition \ref{prop:JM}]The action of $\ad e$ on $\End(M)$ has the hard Lefschetz property for the grading defined by $\ad h$. In particular
\[
	(\ad e)^2 \colon \End(M)_{-2} \isomto \End(M)_2
\]
is an isomorphism. It sends $f$ to $-2e$, so that $f$ is indeed uniquely determined. 

If $e$ and $h$ lie in $\fg$, then $\fg \subset \End(M)$ is graded and the above map restricts to an injective map
\[
	(\ad e)^2 \colon \fg_{-2} \injto \fg_2.
\]
Since $h$ is diagonisable, it is contained in a Cartan sub-algebra of $\fg$. The symmetry of the resulting root system implies that $\dim \fg_{-n} = \dim \fg_n$ for all $n$. In particular, the map $(\ad e)^2$ defines an isomorphism between $\fg_{-2}$ and $\fg_2$, and we conclude that $f$ lies in $\fg$.
\end{proof}

Let $\fa$ be an abelian Lie algebra and $e\colon \fa \to \fgl(M),\, a \mapsto e_a$ a Lie homomorphism. We say that $e$  \emph{has the hard Lefschetz property} if $e(\fa) \subset \fgl(M)_2$ and if there exists some $a\in \fa$ so that $e_a$ has the hard Lefschetz property. Note that this is a Zariski open condition on $a\in\fa$. 

If $e\colon \fa \to \fgl(M)$ has the hard Lefschetz property, then we denote by $\fg(\fa,M)$ the Lie algebra generated by the $\fsl_2$-triples $(e_a,h,f_a)$ for $a\in \fa$ such that $e_a$ has the hard Lefschetz property.  We say that $(\fa,M)$ is a \emph{Lefschetz module} if $\fg(\fa,M)$ is semisimple. 
 
Now let $X$ be a smooth projective complex variety of dimension $d$. Denote by $M := \rH(X,\bQ)[d]$ the shifted total cohomology of $X$ (with middle cohomology in degree $0$). For a class $\lambda \in \rH^2(X,\bQ)$ consider the endomorphism $e_\lambda \in \End(M)$ given by cup product with $\lambda$. If $\lambda$ is ample, then $e_\lambda$ has the hard Lefschetz property, so that the map $e\colon \rH^2(X,\bQ) \to \fgl(M)$ has the hard Lefschetz property. We denote the corresponding Lie algebra by $\fg(X) := \fg(\rH^2(X,\bQ), M)$. 

\begin{proposition}[{\cite[1.6, 1.9]{LooijengaLunts97}}] $(\rH^2(X,\bQ),M)$ is a Lefschetz module.\qed
\end{proposition}

In other words, $\fg(X)$ is a semisimple Lie algebra over $\bQ$.

\subsection{Hochschild homology and cohomology}  Let $X$ be a smooth projective variety of dimension $d$ with canonical bundle $\omega_X := \Omega_X^d$. Its \emph{Hochschild cohomology} is defined as
\[
	\rHH^n(X) := \Ext^n_{X\times X}(\Delta_\ast \cO_X,\Delta_\ast \cO_X)
\]
and its \emph{Hochschild homology} is defined as
\[
	\rHH_n(X) := \Ext^{d-n}_{X\times X}(\Delta_\ast \cO_X, \Delta_\ast \omega_X).
\]
Composition of extensions defines maps $\rHH^n\otimes \rHH^m \to \rHH^{n+m}$ and
$\rHH^n \otimes \rHH_m \to \rHH_{m-n}$ making $\rHH_\bullet(X)$ into a graded module over the graded ring $\rHH^\bullet(X)$.

The Hochschild--Kostant--Rosenberg isomorphism (twisted by the square root of the Todd class as in \cite{Kontsevich03} and \cite{Caldararu06}) defines isomorphisms
\[
	I^n\colon \rHH^n(X) \longisomto \bigoplus_{i+j=n} \rH^{i}(X,\wedge^j T_X)
\]
and
\[
	I_n\colon \rHH_n(X) \longisomto \bigoplus_{j-i=n} \rH^{i}(X,\Omega_X^j).
\]
Under these isomorphisms, multiplication in $\rHH^\bullet(X)$ corresponds to the operation induced by the product in $\wedge^\bullet T_X$, and the action of $\rHH^\bullet(X)$ on $\rHH_\bullet(X)$ corresponds to the action induced by the contraction action of $\wedge^\bullet T_X$ on $\Omega_X^\bullet$, see \cite{CVdB,CRVdB}.

Together with the degeneration of the Hodge--de Rham spectral sequence, the isomorphism $I_\bullet$ defines an isomorphism
\[
	\rHH_\bullet(X) \longisomto \rH(X,\bC).
\]
This map does not respect the grading, rather it maps $\rHH_{i}$ to the $i$-th column of the Hodge diamond (normalised so that the $0$-th column is the central column $\oplus_p \rH^{p,p}$).
Combining with the action of $\rHH^\bullet$ on $\rHH_\bullet$ we obtain an action of the ring $\rHH^\bullet(X)$ on $\rH(X,\bC)$. 

\begin{theorem}\label{thm:hochschild-invariance}
Let $\Phi\colon \cD X_1 \isomto \cD X_2$ be a derived equivalence between smooth projective complex varieties. Then we have natural  graded isomorphisms
\[
	\Phi^{\rHH^\bullet} \colon \rHH^\bullet(X_1) \longisomto \rHH^\bullet(X_2),\quad
	\Phi^{\rHH_\bullet} \colon \rHH_\bullet(X_1) \longisomto \rHH_\bullet(X_2),
\]
compatible with the ring structure on $\rHH^\bullet$ and the module structure on $\rHH_\bullet$, and such that the square
\[
\begin{tikzcd}
	\rHH_\bullet(X_1) \arrow{r}{I} \arrow{d}{\Phi^{\rHH_\bullet}} & \rH(X_1,\bC) \arrow{d}{\Phi^\rH} \\
	\rHH_\bullet(X_2) \arrow{r}{I}  &  \rH(X_2,\bC)
\end{tikzcd}
\]
commutes.
\end{theorem}

\begin{proof}See \cite{CRVdB} and \cite{MacriStellari09}.
\end{proof}

\subsection{The Hochschild Lie algebra of a holomorphic symplectic variety}\label{subsec:bi-graded-lie}

Now assume that $X$ is holomorphic symplectic  of dimension $2d$. That is, we assume that there exists a symplectic form $\sigma \in \rH^0(X,\Omega^2_X)$. Note that this implies that a Zariski dense collection of $\sigma \in \rH^0(X,\Omega^2_X)$ will be nowhere degenerate. 

Through the isomorphism $I\colon \rHH_\bullet(X) \to \rH(X,\bC)$, the vector space $\rH(X,\bC)$ becomes a module under the ring $\rHH^\bullet(X)$. 

\begin{lemma} \label{lemma:symplectic-mirror}
$\rHH^\bullet(X) \cong \rH^\bullet(X,\bC)$ as graded rings, and $\rH(X,\bC)$ is free of rank one as $\rHH^\bullet(X)$-module.
\end{lemma}

\begin{proof}
A symplectic form $\sigma$  defines an isomorphism $\Omega^1_X \isomto T_X$, and hence an isomorphism of algebras $\wedge^\bullet \Omega^1_X \isomto \wedge^\bullet T_X$. Combining this with the Hochschild--Kostant--Rosenberg isomorphism $I$ and the degeneration of the Hodge--de Rham spectral sequence we obtain a chain of isomorphisms of graded rings
\[
	\rHH^\bullet(X) \longisomto \rH^\bullet(X,\wedge^\bullet T_X) 
	\longisomto \rH^\bullet(X,\Omega^\bullet_X) \longisomto \rH^\bullet(X,\bC).
\]
This proves the first assertion. For the second it suffices to observe that the module $\rHH_\bullet(X,\bC)$ is generated by $\sigma^{d} \in \rHH_{2d}(X) = \rH^0(X,\Omega^{2d}_X)$.
\end{proof}

Consider the endomorphisms $h_p,h_q \in \End(\rH(X,\bC))$ given by
\[
	h_p = p-d,\quad h_q = q-d,\quad \text{ on } \rH^{p,q}.
\]
These define the Hodge bi-grading on $\rH(X,\bC)$, normalised to be symmetric along the central part $\rH^{d,d}$. Note that $h=h_p+h_q$.  The action of $\rHH^n(X)$ on $\rH(X,\bC)$ has degree $n$ for the  grading defined by $h'=h_q-h_p$. 

Lemma~\ref{lemma:symplectic-mirror} and Hard Lefschetz imply:

\begin{corollary}For a Zariski dense-collection of $\mu \in \rHH^2(X)$ the action by $\mu$ 
\[
	e'_\mu\colon \rH(X,\bC) \to \rH(X,\bC)
\]
has the hard Lefschetz property with respect to the grading defined by $h'$. \qed
\end{corollary}

In particular, for every such $\mu$ we have a complex subalgebra $\fg_\mu \subset \End(\rH(X,\bC))$ isomorphic to $\fsl_2$, and the collection of such algebras generates a Lie algebra which we denote by $\fg'(X) \subset \End(\rH(X,\bC))$.  From Lemma~\ref{lemma:symplectic-mirror} we also obtain:

\begin{corollary}\label{cor:abstract-comparison}
The complex Lie algebras $\fg'(X)$and $\fg(X)\otimes_\bQ \bC$ are isomorphic.\qed
\end{corollary} 

In the next section, we will show something stronger: that $\fg'(X)$ and $\fg(X)\otimes_\bQ \bC$ coincide as sub-Lie algebras of $\End(\rH(X,\bC))$. Theorem~\ref{bigthm:Lie-invariant} then follows by combining this with the following proposition.

\begin{proposition}\label{prop:derived-invariance-HH-Lie}
Assume that $X_1$ and $X_2$ are holomorphic symplectic varieties. Then for every equivalence $\Phi\colon \cD X_1 \isomto \cD X_2$ there exists a canonical isomorphism of complex Lie algebras
\[
	\Phi^{\fg'}\colon \fg'(X_1) \isomto \fg'(X_2).
\]
It has the property that the map $\Phi^\rH \colon \rH(X_1,\bC) \isomto \rH(X_2,\bC)$ is equivariant with respect to $\Phi^{\fg'}$.
\end{proposition}

\begin{proof}This follows immediately from Theorem~\ref{thm:hochschild-invariance}.
\end{proof}

\subsection{Comparison of the two Lie algebras and proof of Theorem \ref{bigthm:Lie-invariant}}\label{subsec:Lie-invariant}

 The remainder of this section is devoted to the proof of the following.

\begin{proposition}\label{prop:comparison}
If $X$ is holomorphic symplectic, then $\fg(X) \otimes_\bQ \bC = \fg'(X)$ as sub-Lie algebras of $\End(\rH(X,\bC))$. 
\end{proposition}

Let $X$ be holomorphic symplectic. If $\cF$ is a coherent $\cO_X$-module then we will simply write $\rH^i(\cF)$ for $\rH^i(X,\cF)$. We have decompositions
\[
	\rH^2(X,\bC) = \rH^2(\cO_X) \oplus \rH^1(\Omega^1_X) \oplus \rH^0(\Omega^2_X)
\]
and
\[
	\rHH^2(X) = \rH^2(\cO_X) \oplus \rH^1(T_X) \oplus \rH^0(\wedge^2 T_X).
\]
We will use the same symbol $\lambda$ to denote an element $\lambda \in \rH^2(X,\bC)$ and the endomorphism  of $\End(\rH(X,\bC))$ given by cup product with $\lambda$. Note that we have $\lambda \in \fg(X)\otimes_\bQ \bC$ by construction. Similarly, we will use the same symbol for $\mu \in \rHH^2(X)$ and the resulting $\mu  \in \End(\rH(X,\bC))$, given by contraction with $\mu$. We have $\mu \in \fg'(X)$. 

For a  symplectic form  $\sigma \in \rH^0(\Omega^2_X)$ we denote by $\check{\sigma} \in \rH^0(\wedge^2 T_X)$  the image of the form $\sigma \in \rH^0(\Omega^2_X)$ under the isomorphism $\Omega^2_X \to \wedge^2 T_X$ defined by $\sigma$. In suitable local coordinates, 
we have 
\[
	\sigma = \rmd u_1 \wedge \rmd v_1 + \cdots + \rmd u_d \wedge \rmd v_d
\]
and
\[
	\check\sigma = \frac{\partial}{\partial u_1} \wedge \frac{\partial}{\partial v_1}  + \ldots +
	\frac{\partial}{\partial u_d} \wedge \frac{\partial}{\partial v_d}. 
\]

\begin{lemma}\label{lemma:sigma-triple}If $\sigma$ is a nowhere degenerate symplectic form then 
$(\sigma,h_p,\check\sigma)$ is an $\fsl_2$-triple in $\End(\rH(X,\bC))$.
\end{lemma}

\begin{proof}
Clearly $\sigma$ has degree $2$ and $\check\sigma$ has degree $-2$ for the grading given by $h_p$,
so that $[h_p,\sigma]=2\sigma$ and $[h_p,\check\sigma]=-2\check\sigma$.

We need to show that $[  \sigma,  \check\sigma ] = h_p$. This follows immediately from a local computation: in the above local coordinates, one verifies that on the standard basis of $\Omega^p$ the commutator $[\sigma,\check\sigma]$ acts as $p-d$. 
\end{proof}

Note that the existence of one nowhere degenerate $\sigma$ implies that a Zariski dense collection of $\sigma \in \rH^0(\Omega^2_X)$ is nowhere degenerate.

\begin{lemma}\label{lemma:alpha-triple}
For a Zariski-dense collection $\alpha \in \rH^2(X,\cO_X)$ there is a $\check\alpha \in \End(\rH(X,\bC))$ so that $(\alpha,h_q,\check\alpha)$ is an $\fsl_2$-triple.
\end{lemma}

\begin{proof}This follows from Lemma~\ref{lemma:sigma-triple} and Hodge symmetry.
\end{proof}

\begin{lemma}\label{lemma:tau-in-g}
For all $\tau \in \rH^0(X,\wedge^2 T_X)$ the endomorphism $\tau$ lies in $\fg(X)\otimes_\bQ\bC$.
\end{lemma}

\begin{proof}
It suffices to show that this holds for a Zariski dense collection of $\tau$, hence we may assume without loss of generality that $\tau = \check \sigma$ with $\sigma$ and $\check \sigma$ as in Lemma~\ref{lemma:sigma-triple}. Let $\alpha$ and $\check\alpha$ be as in Lemma~\ref{lemma:alpha-triple}. Because $\sigma$ and $h_p$ commute with both $\alpha$ and $h_q$, we have that every element of the $\fsl_2$-triple $(\sigma,h_p,\check \sigma)$ commutes with every element of the $\fsl_2$-triple $(\alpha,h_q,\check \alpha)$. From this, it follows that
\[
	(\alpha + \sigma, h, \check\alpha + \check\sigma)\quad \text{and}\quad (\alpha-\sigma, h, \check\alpha - \check\sigma)
\]
are $\fsl_2$-triples. Since the elements $\alpha \pm \sigma$ lie in $\rH^2(X,\bC)$, and apparently have the hard Lefschetz property, we conclude that the endomorphisms $\check\alpha \pm \check\sigma$ lie in $\fg(X)\otimes_\bQ \bC$, hence also $\tau=\check\sigma$ lies in $\fg(X)\otimes_\bQ \bC$. 
\end{proof}

\begin{corollary}\label{cor:h-in-g}
$h_p$ and $h_q$ lie in $\fg(X)\otimes_\bQ\bC$.
\end{corollary}

\begin{proof}
By Lemma~\ref{lemma:sigma-triple} we have $h_p=[\sigma,\check\sigma]$, which by Lemma~\ref{lemma:tau-in-g} lies in $\fg(X)\otimes_\bQ \bC$. Since $h_q=h-h_p$ we also have that $h_q$ lies in $\fg(X)\otimes_\bQ \bC$.
\end{proof}

Fix a $\tau\in \rH^0(X,\wedge^2 T_X)$ that is nowhere degenerate as an alternating form on $\Omega^1_X$. This defines isomorphisms
$c_\tau \colon \Omega^1_X \to T_X$ and $c_\tau\colon \rH^1(\Omega^1_X)\to\rH^1(T_X)$ given by contracting sections of $\Omega^1_X$ with $\tau$. 

\begin{lemma}\label{lemma:degree-1-bracket}
For all $\eta \in \rH^1(\Omega^1_X)$ we have $[\tau,\eta] = c_\tau(\eta)$ in $\End(\rH(X,\bC))$.
\end{lemma}

\begin{proof}
This is again a local computation. If $\eta$ is a local section of $\Omega^1_X$, then a computation on a local basis shows that
$[\tau,\eta]=c_\tau(\eta)$ as maps $\Omega^p_X \to \Omega^{p-1}_X$. 
\end{proof}

\begin{corollary}\label{cor:eta-in-g}
Every element $\eta'$ of $\rH^1(X,T_X)$ lies in $\fg(X)\otimes_\bQ \bC$.
\end{corollary}

\begin{proof}
Every such $\eta'$ is of the form $c_\tau(\eta)$ for a unique $\eta \in \rH^1(\Omega^1_X)$, and hence the corollary follows from Lemma~\ref{lemma:degree-1-bracket}, Lemma~\ref{lemma:tau-in-g}, and the fact that $\eta$ lies in $\fg(X)\otimes_\bQ \bC$.
\end{proof}

We can now finish the comparison of the two Lie algebras.

\begin{proof}[Proof of Proposition~\ref{prop:comparison}] 
By Corollary~\ref{cor:abstract-comparison} it suffices to show that $\fg'(X)$ is contained in $\fg(X)\otimes_\bQ \bC$. By Proposition~\ref{prop:JM} it suffices to show that $h'$ is contained in $\fg(X)\otimes_\bQ \bC$, and that for almost every $a \in \rHH^2(X)$ we have that the action of $a$ on $\rH(X,\bC)$ is contained in $\fg(X)\otimes_\bQ \bC$. This follows from Lemma~\ref{lemma:tau-in-g}, Corollary~\ref{cor:h-in-g}, and Corollary~\ref{cor:eta-in-g}, and the fact that the action of any $\alpha \in \rH^2(\cO_X)$ lies in $\fg(X)\otimes_\bQ \bC$.
\end{proof}

Together with Proposition~\ref{prop:derived-invariance-HH-Lie}, this proves Theorem \ref{bigthm:Lie-invariant}.

%\begin{theorem}\label{thm:derived-invariance-Lie}
%Assume that $X_1$ and $X_2$ are holomorphic symplectic varieties. Then for every Fourier--Mukai equivalence $\Phi\colon \cD X_1 \to \cD X_2$ there exists a canonical isomorphism of rational Lie algebras
%\[
%	\Phi^\fg\colon \fg(X_1) \to \fg(X_2).
%\]
%It has the property that the map $\Phi^\rH \colon \rH(X_1,\bQ) \to \rH(X_2,\bQ)$ is equivariant with respect to $\Phi^\fg$.
%\end{theorem}

%%%%%%%%%% 

\section{Rational cohomology of hyperk\"ahler varieties}

\subsection{The BBF form and the LLV Lie algebra}\label{sec:LL-so}

Let $X$ be a complex hyperk\"ahler variety of dimension $2d$. We denote by
\[
	b = b_X\colon \rH^2(X,\bQ) \times \rH^2(X,\bQ) \to \bQ
\]
its Beauville--Bogomolov--Fujiki, and by $c_X$ its Fujiki constant. These are related by
\begin{equation}\label{eq:fujiki-relation}
	\int_X \lambda^{2d} = \frac{(2d)!}{2^d d!} c_X b(\lambda,\lambda)^d
\end{equation}
for $\lambda \in \rH^2(X,\bQ)$. 

We extend $b$ to a bilinear form on
\[
	\widetilde\rH(X,\bQ) := \bQ \alpha \oplus \rH^2(X,\bQ) \oplus \bQ \beta,
\]
by declaring $\alpha$ and $\beta$ to be orthogonal to $\rH^2(X,\bQ)$, and setting $b(\alpha,\beta)=-1$, $b(\alpha,\alpha)=0$ and $b(\beta,\beta)=0$. We equip $\widetilde\rH(X,\bQ)$ with a grading satisfying $\deg \alpha =-2$, $\deg \beta =2$, and for which $\rH^2(X,\bQ)$ sits in degree $0$. This induces a grading on the Lie algebra $\fso(\widetilde\rH(X,\bQ))$. 

For $\lambda \in \rH^2(X,\bQ)$ we consider the endomorphism $e_\lambda \in \fso(\widetilde\rH(X,\bQ))$ given by $e_\lambda(\alpha) = \lambda$, $e_\lambda(\mu) = b(\lambda,\mu)\beta$ for all $\mu \in \rH^2(X,\bQ)$, and $e_\lambda(\beta)=0$.

\begin{theorem}[Looijenga--Lunts, Verbitsky]\label{thm:LLV}
There is a unique isomorphism of graded Lie algebras
\[
	\fso(\widetilde\rH(X,\bQ)) \longisomto \fg(X)
\]
that maps $e_\lambda$ to $e_\lambda$ for every $\lambda \in \rH^2(X,\bQ)$.
\end{theorem}

\begin{proof}See \cite[Prop.~4.5]{LooijengaLunts97} or \cite[Thm.~1.4]{Verbitsky96} for the theorem over the real numbers. This readily descends to $\bQ$, see \cite[Prop.~2.9]{Soldatenkov} for more details. 
\end{proof}

The representation of $\fso(\widetilde\rH(X,\bQ))$ on $\rH(X,\bQ)$  integrates to a representation of 
$\Spin(\widetilde\rH(X,\bQ))$ on $\rH(X,\bQ)$. Let $\lambda \in \rH^2(X,\bQ)$. Then $e_\lambda$ is nilpotent, and hence $B_\lambda := \exp e_\lambda$ is an element of $\Spin(\widetilde\rH(X,\bQ))$. It acts on $\widetilde\rH(X,\bQ)$ as follows:
\[
	B_\lambda( r\alpha + \mu + s\beta) = 
	r\alpha + (\mu + r \lambda) + \big(s + b(\mu,\lambda) + r \tfrac{b(\lambda,\lambda)}{2}\big)\beta
\]
for all $r,s \in \bQ$ and $\mu \in \rH^2(X,\bQ)$. The action on the total cohomology of $X$ is given by:

\begin{proposition}\label{prop:exp-chern}
$B_\lambda$ acts as multiplication by $\ch(\lambda)$ on $\rH(X,\bQ)$. \qed
\end{proposition}

In particular, if $\cL$ is a line bundle on $X$ and $\Phi\colon \cD X \to \cD X$ is the equivalence that maps $\cF$ to $\cF\otimes \cL$, then $\Phi^\rH=B_{c_1(\cL)}$.

\subsection{The Verbitsky component of cohomology}\label{subsec:regular-part}
Let $X$ be a complex hyperk\"ahler variety of dimension $2d$.
We define the \emph{even cohomology} of $X$ as the graded $\bQ$-algebra
\[
	\rH^\ev(X,\bQ) := \bigoplus_{n} \rH^{2n}(X,\bQ),
\]
and the \emph{Verbitsky component} of the cohomology of $X$ as the sub-$\bQ$-algebra $\SH(X,\bQ)$ of~$\rH^\ev(X,\bQ)$ generated by $\rH^2(X,\bQ)$. Clearly, $\SH(X,\bQ)[2d]$ is a sub-Lefschetz-module of $\rH^\ev(X,\bQ)[2d]$ for $\rH^2(X,\bQ)$.  

\begin{lemma}[Verbitsky \cite{Verbitsky95,Bogomolov96}]
The kernel of the $\bQ$-algebra homomorphism
\[
	\Sym^\bullet \rH^2(X,\bQ) \surjto \SH(X,\bQ)
\]
is generated by the elements $\lambda^{d+1}$ with $\lambda\in \rH^2(X,\bQ)$ satisfying
$b(\lambda,\lambda)=0$.\qed
\end{lemma}

\begin{lemma}[Verbitsky]\label{lemma:verbitsky}
$\SH(X,\bQ)[2d]$ is an irreducible Lefschetz module.
\end{lemma}

\begin{proof}It is the smallest sub-Lefschetz module of $\rH^\ev(X,\bQ)[2d]$ having a non-trivial component of degree $-2d$. 
\end{proof}

Verbitsky also describes the space $\SH(X,\bQ)$ explicitly. Below we normalise this description, and use it to compute the Mukai pairing on $\SH(X,\bQ)$.

\begin{proposition}\label{prop:psi}
There is a unique map
\[
	\Psi\colon \SH(X,\bQ)[2d] \longto \Sym^d \widetilde\rH(X,\bQ)
\]
satisfying 
\begin{enumerate}
\item $\Psi$ is morphism of Lefschetz modules
\item $\Psi(1) = \alpha^d/d!$
\end{enumerate}
\end{proposition}

Note that the Lefschetz module structure on $\Sym^d \widetilde\rH(X,\bQ)$ is given by the Leibniz rule
\[
	e_\lambda(x_1\cdots x_d) := \sum_i x_1\cdots e_\lambda(x_i) \cdots x_d.
\]

\begin{proof}[Proof of Proposition~\ref{prop:psi}]
Uniqueness is clear. For existence, consider the map
\[
	\widetilde\Psi\colon \Sym^\bullet \rH^2(X,\bQ) \longto \Sym^d \widetilde\rH(X,\bQ),
\]
given by
\[
	\lambda_1\cdots \lambda_n \mapsto e_{\lambda_1} \cdots e_{\lambda_n}(\alpha^d/d!).
\]
This map is well-defined since the $e_{\lambda_i}$ commute, and by construction it is a morphism of Lefschetz modules satisfying $\Psi(1)=\alpha^d/d!$. It suffices to show that $\widetilde{\Psi}$ vanishes on the ideal generated by the $\lambda^{d+1}$ for $\lambda \in \rH^2(X,\bQ)$ satisfying $b(\lambda,\lambda)=0$. Equivalently, it suffices to show that for every $x\in \Sym^d \widetilde\rH(X,\bQ)$ and for every $\lambda \in \rH^2(X,\bQ)$ with $b(\lambda,\lambda)=0$ we have $e_\lambda^{d+1}(x)=0$. 

Without loss of generality, we may assume that $x$ is a monomial of the form
\[
	x= \alpha^i\beta^j \lambda_1\cdots \lambda_m,\quad i+j+m=d,\, \lambda_i \in \rH^2(X,\bQ).
\]
By degree reasons, we have $e_\lambda^k(\beta^j \lambda_1\cdots \lambda_m)=0$ for $k>m$. Moreover, it follows from $b(\lambda,\lambda)=0$ that $e_\lambda^k(\alpha^i)=0$ for $k>i$. Combining these, one concludes that $e_\lambda^{d+1}(x)=0$, which is what we had to prove.
\end{proof}

\begin{lemma}$\Psi(\pt_X) = \beta^d/c_X$.
\end{lemma}
\begin{proof}
Choose $\lambda \in \rH^2(X,\bQ)$ with $b(\lambda,\lambda)\neq 0$. Then we have
\begin{equation}\label{eq:Psi-pt}
	\Psi(\lambda^{2d})=e_\lambda^{2d}(\frac{\alpha^d}{d!}) = 
	\frac{(2d)!}{2^d d! } b(\lambda,\lambda)^d \beta^d.
\end{equation}
Dividing by  (\ref{eq:fujiki-relation}) gives the claimed identity.  
\end{proof}

Consider the contraction (or Laplacian) operator
\[
	\Delta\colon \Sym^d \widetilde\rH(X,\bQ) \to \Sym^{d-2}  \widetilde\rH(X,\bQ),
\]
given by
\[
	x_1\ldots x_d \mapsto \sum_{i<j} b(x_i,x_j) x_1\cdots \hat{x}_i \cdots \hat{x}_j \cdots x_d.
\]
This is a morphism of Lefschetz modules, or equivalently of $\fso(\widetilde{\rH}(X,\bQ))$-modules.

\begin{lemma}\label{lemma:exact-sequence}
The sequence of Lefschetz modules
\[
	0 \longto \SH(X,\bQ)[2d] \overset{\Psi}{\longto}  \Sym^d \widetilde\rH(X,\bQ)  \overset{\Delta}{\longto}
	\Sym^{d-2}  \widetilde\rH(X,\bQ) \longto 0
\]
is exact.
\end{lemma}

\begin{proof}Since $\Delta\Psi(1)=0$, we have $\Delta\circ \Psi=0$. The map $\Delta$ is well-known to be a surjective map of $\fso(\widetilde\rH(X,\bQ))$-modules with irreducible kernel. Since $\Psi$ is non-zero and $\SH(X,\bQ)$ is irreducible, it follows that the sequence is exact. 
\end{proof}

The \emph{Mukai pairing} \cite{Caldararu03} on $\rH^\ev(X,\bQ)$ restricts to a pairing $b_\SH$ on $\SH(X,\bQ)$. It pairs elements of degree $m$ with elements of degree $2d-m$, according to the formula
\[
	b_\SH(\lambda_1\cdots \lambda_m, \,\mu_1\cdots \mu_{2d-m} ) =
	(-1)^m \int_X \lambda_1 \cdots \lambda_m \mu_1\cdots \mu_{2d-m}.
\]
Note that $b_\SH(e_\lambda x,y) + b_\SH(x,e_\lambda y)=0$ for all $x,y \in \SH(X,\bQ)$ and $\lambda \in \rH^2(X,\bQ)$, so that $b_\SH$ is $\fso(\widetilde\rH(X,\bQ))$-invariant.

The pairing on $\widetilde\rH(X,\bQ)$ induces a pairing on $\Sym^d \widetilde\rH(X,\bQ)$ defined by
\[
	b_{[d]}(x_1\cdots x_d, y_1\cdots y_d) := 
	(-1)^d\sum_{\sigma \in \fS_d} \prod_i b(x_i,y_{\sigma i}).
\]
By construction, $b_{[d]}$ is $\fso(\widetilde\rH(X,\bQ))$-invariant. The map $\Psi$ is almost
an isometry, in the following sense.

\begin{proposition}\label{prop:bbf-on-sh}
For all $x,y \in \SH(X,\bQ)$ we have
\[
	b_{[d]}(\Psi x, \Psi y) = \frac{2^d d!}{c_X (2d)!} b_{\SH}(x,y).
\]
\end{proposition}

\begin{proof}
Both the Mukai form on $\SH(X,\bQ)[2d]$ and the pairing on $\Sym^d \widetilde\rH(X,\bQ)$ are $\fso(\widetilde\rH(X,\bQ))$-invariant. Since $\SH(X,\bQ)$ is an irreducible $\fso(\widetilde\rH(X,\bQ))$-module, it suffices to verify the identity for some $x,y\in \SH(X,\bQ)$ with $b_{\SH}(x,y)\neq 0$.

Let $\lambda \in \rH^2(X,\bQ)$ with $b(\lambda,\lambda)\neq 0$. We have
\[
	b_\SH(1,\lambda^{2d}) = \int_X \lambda^{2d} =  c_X b(\lambda,\lambda)^d \neq 0.
\]	
By (\ref{eq:Psi-pt}) we have
\[
	\Psi(\lambda^{2d})= 
	\frac{(2d)!}{2^d d! } b(\lambda,\lambda)^d \beta^d,
\]
%(For example, one verifies that $e_\lambda^{2d}(\alpha^{d})=\tfrac{2d(d-1)}{2} e_\lambda^{2d-2}(\alpha^{d-1}) b(\lambda,\lambda) \beta$, and concludes by induction.)
and hence
\[
	b_{[d]}(\Psi(1),\Psi(\lambda^{2d})) = \frac{(2d)!}{ 2^d(d!)^2} b(\lambda,\lambda)^d b_{[d]}(\alpha^d,\beta^d)
	=   \frac{(2d)!}{2^d d!} b(\lambda,\lambda)^d,
\]
which agrees with the identity claimed in the proposition.
\end{proof}

\begin{remark}
If $X$ is of type $\mathrm{K3}^{[d]}$ then $c_X = (2d)!/(2^d d!)$ and $\Psi$ is an isometry.
\end{remark}

\section{Action of derived equivalences on the Verbitsky component}

In this section we prove Theorems~\ref{bigthm:hyperkahler-subring} and  \ref{bigthm:normalizer} from the introduction.

\subsection{A representation-theoretical construction}
Let $K$ be a field of characteristic different from $2$. Let $V=(V,b)$ be a non-zero quadratic space over $K$. Let $d$ be a positive integer and consider the space
\[
	S_{[d]} V := \ker (\Sym^d V \overset{\Delta}{\longto} \Sym^{d-2} V).
\]
The Lie algebra $\fso(V)$ acts faithfully on $S_{[d]} V$, inducing an inclusion $\fso(V) \subset \End(S_{d}V)$. Consider the normalizer of $\fso(V)$ in $\GL(S_{[d]} V)$, that is, the group
\[
	N(V,d) := \big\{\, g\in \GL(S_{[d]} V) \,\mid\, 
	g \fso(V) g^{-1} = \fso(V)  \big\}.
\]

\begin{proposition}\label{prop:normalizer-computation}
There is an exact sequence
\[
 	1 \longto \{\pm 1\} \longto \rO(V) \times K^\times \longto N(V,d) \longto 1
\]
where the inclusion maps $\epsilon$ to $(\epsilon, \epsilon^d)$ and the
surjection maps $(\varphi,\lambda)$ to $\lambda S_{[d]}(\varphi)$.
\end{proposition}

\begin{proof}
The only non-trivial part is surjectivity of $\rO(V)\times K^\times \to N(V,d)$. Denote by $\varphi\colon \rO(V)\to N(V,d)$ the restriction of this map to the first component.

The representation $S_{[d]} V$ of $\fso(V)$ is irreducible, so by  Schur's lemma the centralizer of $\fso(V)$ in $\GL(S_{[d]} V)$ is $K^\times$, and we have an exact sequence
\[
	1 \longto K^\times \longto N(V,d) \overset{\psi}\longto \Aut(\fso(V)).
\]
It therefore suffices to show that the image of $\psi$ equals the image of $\psi\circ \varphi$.

The adjoint group of $\fso(V)$ is $\SO(V)$, and we have a short exact sequence
\[
	1 \longto \SO(V) \longto \Aut(\fso(V)) \longto \Out(\fso(V)) \longto 1,
\]
where $\Out(\fso(V))$ coincides with the group of symmetries of the Dynkin diagram. 

%Note that this map is injective if $d$ is odd, and has kernel $\{\pm 1\}$ if $d$ is even.

If $\dim V=2n+1$, then we have $\rO(V) = \{\pm 1\} \times \SO(V)$. The Dynkin diagram (of type $B_n$) has no non-trivial automorphisms, so
$\Aut(\fso(V,b)) = \SO(V)$. The composition $\psi\circ \varphi$  maps $\SO(V)$ identically to $\SO(V)$, and we conclude that the image of $\psi$ is the image of $\psi\circ \varphi$.

If $m=2n$, then  we have $\rO(V)/\{\pm 1\} \subset \Aut \fso(V)$, with elements of determinant $-1$ inducing the reflection in the horizontal axis in the Dynkin diagram (of type $D_n$). For $n\neq 4$, this inclusion is an equality, while for $n=4$ `triality' gives extra automorphisms. However, expressed on simple roots the highest weight of the representation $S_{[d]} V$ of $\fso(V)$ is
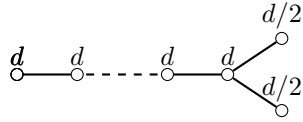
\[
  \begin{tikzpicture}[scale=.8]
  %    \draw (-1,0) node[anchor=east]  {$(m=2n)$};
    \foreach \x in {0,...,2}
    \draw (-0.5,0) circle (.1) node[anchor=south] {$d$};
    \draw (.5,0) circle (.1) node[anchor=south] {$d$};
    \draw (2,0) circle (.1) node[anchor=south] {$d$};
    \draw (3,0) circle (.1) node[anchor=south] {$d$};
    \draw (3.9,0.6) circle (.1) node[anchor=south] {$d/2$};
    \draw (3.9,-0.6) circle (.1) node[anchor=south] {$d/2$};
    \draw[thick] (-0.4,0) -- (0.4,0);
    \draw[thick,dashed] (0.65,0) -- (1.9,0);
    \draw[thick] (2.1,0) -- (2.9,0);
    \draw[thick] (3.09,0.06) -- (3.81,0.54);
    \draw[thick] (3.09,-0.06) -- (3.81,-0.54);
  \end{tikzpicture}
\]
so that for $n=4$ the extra automorphisms of $\fso(V)$ do not lift to automorphisms of $S_{[d]} V$. We conclude that the image of $\psi$ is contained in $\rO(V)/\{\pm 1\}$ and that the composition $\psi\circ \varphi$ is the natural map $\rO(V)\to \rO(V)/\{\pm 1\}$, so that also in this case the image of 
$\psi$ coincides with the image of $\psi\circ \varphi$.
\end{proof}

\begin{proposition}\label{prop:twist-similitude}
Let $V_1$ and $V_2$ be non-zero quadratic spaces. Assume that there is a linear map
$f\colon S_{[d]} V_1 \to S_{[d]} V_2$ such that $f \fso(V_1) f^{-1} = \fso(V_2)$ as subspaces of
 $\End(V_2)$. Then there exists a  $\mu \in K^\times$ and a similitude $\varphi\colon V_1 \to V_2$ such that $f=\mu S_{[d]}(\varphi)$.
\end{proposition}

\begin{proof}
Let $\bar K$ be a separable closure of $K$. After base change to $\bar K$ the quadratic spaces $V_1$ and $V_2$ become isometric, hence
$V_2$ determines a class $\gamma$ in the  Galois cohomology group $\rH^1(\Gal(\bar K/K), \rO(V_{\bar K})$. The existence of $f$ shows
that $\gamma$ is mapped to the trivial element under  the natural map
\[
	\theta\colon \rH^1\big(\Gal(\bar K/K),\, \rO(V_{1,\bar K}) \big)
	\longto
	\rH^1\big(\Gal(\bar K/K),\,N(V_{1,\bar K}, d) \big).
\]
By Proposition~\ref{prop:normalizer-computation} (and Hilbert 90), this shows that
$\gamma$ is in the image of the map
\[
	K^\times/(K^\times)^2 = \rH^1(\Gal(\bar{K}/K), \{\pm 1\} ) \to \rH^1\big(\Gal(\bar K/K),\, \rO(V_{1,\bar K}) \big)
\]
induced by $\{\pm 1\} \to \rO(V_{1,\bar K})$. But his means $(V_1,\lambda  b_1)$ and
$(V_2,b_2)$ are isomorphic for any representative $\lambda$ of a class in  $K^\times/(K^\times)^2$ mapping to $\gamma$.

In particular, there exists a similitude $\varphi_0\colon V_1 \to V_2$. The map $f^{-1} \circ S_{[d]}(\varphi_0)$ is an element of $N(V_1,d)$. Now let $(\psi,\mu) \in \rO(V_1) \times K^\times$ be a pre-image of this element under the map from Proposition~\ref{prop:normalizer-computation}, and set
$\varphi := \varphi_0 \circ \psi^{-1}$. Then the pair $(\mu,\varphi)$ satisfies the requirements.
\end{proof}

The bilinear form $b$ on $V$ induces a bilinear form $b_{[d]}$ on $S_{[d]} V$ defined as
\[
	b_{[d]}(x_1\cdots x_d,\, y_1\cdots y_d) := (-1)^d
	\sum_{\sigma \in S_n} \prod_i b(x_i,y_{\sigma i}),
\]
Consider the group
\[
	G(V,b,d) := N(V,d) \cap \rO(S_{[d]}, b_{[d]})
\]
of isometries of $S_{[d]} V$ that preserve the subspace $\fso(V)$ of $\End S_{[d]} V$.

\begin{proposition}\label{prop:functor-aut}Let $V=(V,b)$ be an object of $\cB$. We have
\[
	G(V,b,d) = \begin{cases}
		\rO(V,b) & \text{ $d$ odd, } \\
%		\SO(V,b) \times \{\pm 1\} & \text{ $d$ even and $m$ odd } \\
		\rO(V,b)/\{\pm 1\} \times \{\pm1\}  & \text{ $d$ even.}
	\end{cases}
\]
\end{proposition}

\begin{proof}
This follows immediately from Proposition~\ref{prop:normalizer-computation}.
\end{proof}

\subsection{The Verbitsky component}\label{subsec:verbitsky-component}

\begin{theorem}\label{thm:hyperkahler-subring}
Let $X_1$ and $X_2$ be hyperk\"ahler varieties and $\Phi\colon \cD X_1 \to \cD X_2$ an equivalence. Then the induced isomorphism $\Phi^{\rH} \colon \rH(X_1,\bQ) \to \rH(X_2,\bQ)$ restricts  to an isomorphism $\Phi^{\SH}\colon \SH(X_1,\bQ) \to \SH(X_2,\bQ)$. Moreover
\begin{enumerate}
\item $\Phi^{\SH}$ is an isometry with respect to the Mukai pairings
\item $\Phi^{\SH} \fg(X_1) (\Phi^{\SH})^{-1} = \fg(X_2)$ in $\End(\SH(X_2,\bQ))$.
\end{enumerate}
\end{theorem}

\begin{proof}Note that $\SH(X,\bQ)$ can be characterized as the minimal sub-$\fg(X)$-module of $\rH(X,\bQ)$ whose Hodge structure attains the maximal possible level (width). It then follows from
Theorem \ref{bigthm:Lie-invariant} and from Lemma~\ref{lemma:verbitsky} that $\Phi^{\rH}$ restricts to an isomorphism
\[
	\Phi^{\SH} \colon \SH(X_1,\bQ) \isomto \SH(X_2,\bQ)
\]
respecting the Lie algebras $\fg(X_1)$ and $\fg(X_2)$. By \cite{Caldararu03}, the map $\Phi^\rH$ respects the Mukai pairings, and the theorem follows.
\end{proof}

\begin{definition}\label{def:weight-zero-HS}
For a complex hyperk\"ahler variety we equip $\SH(X,\bQ)$ and $\widetilde\rH(X,\bQ)$ with Hodge structures of weight $0$, given by
\[
	\SH(X,\bQ) \subset \rH^\ev(X,\bQ) = \oplus_n \rH^{2n}(X,\bQ(n))
\]
and
\[
	\widetilde\rH(X,\bQ) = \bQ \alpha \oplus \rH^2(X,\bQ(1)) \oplus \bQ \beta.
\]
\end{definition}

\begin{lemma}Let $X$ be a hyperk\"ahler variety of dimension $2d$. Then the map
\[
	\Psi\colon \SH(X,\bQ) \to \Sym^d \widetilde\rH(X,\bQ)
\]
of Proposition~\ref{prop:psi} is a morphism of Hodge structures of weight~$0$.
\end{lemma}

\begin{proof}
This is clear from the proof of Proposition~\ref{prop:psi}: the map $\widetilde\Psi$ is a morphism of Hodge structures, and so is the quotient map $\Sym^\bullet \rH(X,\bQ(1)) \to \SH(X,\bQ)$.
\end{proof}

\begin{proposition}\label{prop:similitude}
Let $X_1$ and $X_2$ be derived equivalent hyperk\"ahler varieties. Then there exists
a Hodge similitude $\varphi\colon \widetilde\rH(X_1,\bQ) \isomto \widetilde\rH(X_2,\bQ)$ 
and a scalar $\lambda \in K^\times$ so that the square
\[
\begin{tikzcd}[column sep=huge]
	\SH(X_1,\bQ) \arrow{r}{\Phi^{\SH}} \arrow{d}{\Psi} & \SH(X_2,\bQ) \arrow{d}{\Psi} \\
	\Sym^d  \widetilde\rH(X_1,\bQ) 
		\arrow{r}{ \lambda \Sym^d(\varphi)}
	& \Sym^d  \widetilde\rH(X_2,\bQ) 
\end{tikzcd}
\]
commutes.
\end{proposition}

\begin{proof}
By Theorem~\ref{thm:hyperkahler-subring} and Proposition~\ref{prop:twist-similitude} there 
exists a similitude $\varphi$ and a scalar $\lambda$ that make the square commute.

It remains to check that $\varphi$ respects the Hodge structures. The Hodge structure on $\widetilde\rH(X_i,\bQ)$ is given by a morphism $h_i\colon \bC^\times \to \rO(\rH(X_i,\bR))$, and the preceding lemma implies that the Hodge structure on $\SH(X_i,\bQ)$ is given by composing $h_i$ with the injective map $\rO(\rH(X_i,\bR)) \to \GL(\SH(X_i,\bR))$.  Since $\varphi$ maps the Hodge structure on $\SH(X_1,\bQ)$ to the Hodge structure on $\SH(X_2,\bQ)$, we conclude that $\varphi$ maps $h_1$ to $h_2$.
\end{proof}

\begin{theorem}[$d$ odd]\label{thm:odd-mukai-functoriality}
Assume that $d$ is odd, and that $X_1$ and $X_2$ are deformation-equivalent hyperk\"ahler varieties of dimension $2d$. Let  $\Phi\colon \cD X_1 \isomto \cD X_2$ be
an equivalence. Then there is a unique Hodge isometry $\Phi^{\widetilde\rH}$
making the square
\[
\begin{tikzcd}[column sep=huge]\label{thm:isom-odd}
	\SH(X_1,\bQ) \arrow{r}{\Phi^{\SH}} \arrow{d}{\Psi} & \SH(X_2,\bQ) \arrow{d}{\Psi} \\
	\Sym^d  \widetilde\rH(X_1,\bQ) 
		\arrow{r}{ \Sym^d(\Phi^{\widetilde\rH})}
	& \Sym^d  \widetilde\rH(X_2,\bQ) 
\end{tikzcd}
\]
commute. The formation of $\Phi^{\widetilde\rH}$ is functorial 
in $\Phi$.
\end{theorem}

\begin{proof}
Since $X_1$ and $X_2$ are deformation equivalent, we can choose an isometry
$\varphi\colon \widetilde\rH(X_1,\bQ) \isomto \widetilde\rH(X_2,\bQ)$. Moreover, $X_1$ and
$X_2$ have the same Fujiki constant, so $\Sym^d\varphi$ restricts to an isometry between the
images of $\Psi$. 
Then by Theorem~\ref{thm:hyperkahler-subring} and Proposition~\ref{prop:functor-aut}, there is a unique isometry
$\psi \in \rO(\widetilde\rH(X_2,\bQ))$ such that $\Phi^{\widetilde\rH} := \psi\varphi$ makes the square commute. Uniqueness forces its formation to be functorial.

That $\Phi^{\widetilde\rH}$ respects the Hodge structures follows from the same argument
as in the proof of Proposition~\ref{prop:similitude}. 
\end{proof}

If $d$ is even, then the natural map
\[
	\rO( \widetilde\rH(X,\bQ)) \to \Aut (\SH(X,\bQ),b_X,\fg(X))
\]
is neither injective, nor surjective, and the proof above fails. However, if we moreover assume that $b_2(X)$ is odd, then one can use  the isomorphism $\SO\times \{\pm1\} \cong \rO$ to salvage the situation a bit.

Define an \emph{orientation} on $X$ to be the choice of a generator of $\det \rH^2(X,\bR)$, up to $\bR^\times_{>0}$. Equivalently, an orientation is the choice of generator of $\det \widetilde\rH(X,\bR)$ up to $\bR^\times_{>0}$. Define the \emph{sign} $\epsilon(\varphi)$ of a Hodge isometry $\varphi\colon \widetilde\rH(X_1,\bQ) \isomto \widetilde\rH(X_2,\bQ)$ as $\epsilon(\varphi)=1$ if $\varphi$ respects the orientations and $\epsilon(\varphi)=-1$ otherwise. A derived equivalence between oriented hyperk\"ahler varieties is a derived equivalence between the underlying unoriented hyperk\"ahler varieties.  

\begin{theorem}[$d$ even]\label{thm:isom-even}\label{thm:even-mukai-functoriality}
Assume that $d$ is even, and that $\Phi\colon \cD X_1 \isomto \cD X_2$ is a derived equivalence between oriented hyperk\"ahler varieties of dimension $2d$. Assume that $X_1$ and $X_2$ have odd $b_2$, and that the quadratic spaces $\rH^2(X_1,\bQ)$ and $\rH^2(X_2,\bQ)$ are isometric. Then there exists 
a unique  Hodge isometry $\Phi^{\widetilde\rH}$ making  the square 
\[
\begin{tikzcd}[column sep=huge]
	\SH(X_1,\bQ) \arrow{r}{\epsilon(\Phi^{\widetilde\rH})\Phi^{\SH}} \arrow{d}{\Psi} & \SH(X_2,\bQ) \arrow{d}{\Psi} \\
	\Sym^d  \widetilde\rH(X_1,\bQ) 
		\arrow{r}{ \Sym^d(\Phi^{\widetilde\rH})}
	& \Sym^d  \widetilde\rH(X_2,\bQ) 
\end{tikzcd}
\]
commute. Morever, the formation of $\Phi^{\widetilde\rH}$ is functorial for composition of derived equivalences between  hyperk\"ahler varieties equipped with orientations.
\end{theorem}

\begin{proof}
This follows from Theorem~\ref{thm:hyperkahler-subring} and Proposition~\ref{prop:functor-aut} with essentially the same argument as the proof of Theorem~\ref{thm:isom-odd}.
\end{proof}

\begin{remark}
If $X_1$ and $X_2$ are hyperk\"ahler varieties belonging to one of the known families, and if $\Phi\colon \cD X_1 \isomto \cD X_2$ is an equivalence, then the hypotheses of either Theorem~\ref{thm:isom-odd} or Theorem~\ref{thm:isom-even} are satisfied. Indeed, $X_1$ and $X_2$ will have the same dimension $2d$ and because they have isomorphic LLV Lie algebra, they have the same second Betti number $b_2$. Going through the list of known families, one sees that this implies that $X_1$ and $X_2$ are deformation equivalent. In particular, they have isometric $\rH^2$. Finally, all known hyperk\"ahler varieties of dimension $2d$ with $d$ even have odd $b_2$.  
\end{remark}

Taking $X_1=X_2$ in Theorem~\ref{thm:isom-odd} and Theorem~\ref{thm:isom-even} yields Theorem~\ref{bigthm:normalizer} from the introduction:

\begin{theorem}\label{thm:normalizer}
Let $X$ be a hyperk\"ahler variety of dimension $2d$. Assume that either $d$ is odd or that $d$ is even and $b_2(X)$ is odd. Then the representation $\rho^\SH \colon \Aut \cD(X) \to \GL(\SH(X,\bQ))$ factors over
a  map $\rho^{\widetilde\rH} \colon \Aut \cD(X) \to \rO(\widetilde\rH(X,\bQ))$.
\end{theorem}

\begin{remark}\label{rmk:explicit-normalizer}
For $d$ odd, the implicit map $\rO(\widetilde\rH(X,\bQ)) \to \GL(\SH(X,\bQ))$ is the natural map coming from the isomorphism $\SH(X,\bQ) \cong S_{[d]} \widetilde\rH(X,\bQ))$. For $d$ even (and $b_2$ odd), it is the twist of the natural map with the determinant character $\rO(\widetilde\rH(X,\bQ)) \to \{\pm 1\}$. 
\end{remark}

\section{Hodge structures}\label{sec:hodge}

In this section we prove Theorem~\ref{bigthm:hodge-structures} from the introduction. 

For a rational quadratic space $V$ we will make use of the algebraic groups $\bSO(V)$, $\bSpin(V)$, and $\bGSpin(V)$ over $\Spec \bQ$. These groups sit in exact sequences 
\begin{gather}
	\label{eq:def-so} 1 \longto \{\pm 1\} \longto \bSpin(V) \longto \bSO(V) \longto 1 \\
	\label{eq:gspin-so} 1 \longto \bG_m \longto \bGSpin(V) \longto \bSO(V) \longto 1 \\
	\label{eq:spin-gspin} 1 \longto \{\pm 1\} \longto \bG_m \times \bSpin(V) \longto \bGSpin(V) \longto 1.
\end{gather}
We will write $\SO(V)$, $\Spin(V)$, and $\GSpin(V)$ for the groups of $\bQ$-points of these algebraic groups.

\begin{lemma}\label{lemma:gspin-acts}
Let $X$ be a hyperk\"ahler variety. There exists a unique action of $\bGSpin(\widetilde\rH(X,\bQ))$ on $\rH(X,\bQ)$ such that
\begin{enumerate}
\item the action of $\bSpin(\widetilde\rH(X,\bQ)) \subset \bGSpin(\widetilde\rH(X,\bQ))$ integrates the action of $\fg(X)=\fso(\widetilde\rH(X,\bQ))$
\item a section $\lambda \in \bG_m \subset \bGSpin(\widetilde\rH(X,\bQ))$ acts as $\lambda^{i-2d}$ on $\rH^i(X,\bQ)$.
\end{enumerate}
\end{lemma}

\begin{proof}
The action of $\fso(\widetilde\rH(X,\bQ))$ integrates to an action of the simply connected group $\bSpin(\widetilde\rH(X,\bQ))$. By (\ref{eq:spin-gspin}) it suffices to verify that $-1\in \Spin(\widetilde\rH(X,\bQ))$ acts as $(-1)^i$ on $\rH^i(X,\bQ)$. Any $\fsl_2$-triple $(e_\lambda,h,f_\lambda)$ in $\fg(X)$ induces an action of $\SL_2$ on $\rH(X,\bQ)$, with the property that $\diag(\lambda,\lambda^{-1}) \in \SL_2(\bQ)$ acts as $\lambda^i$ on $\rH^{2d+i}(X,\bQ)$. Since the inclusion $\bSL_2 \subset \bSpin$ maps $\diag(1,-1)$ to $-1$, the lemma follows.
\end{proof}

Recall from Definition~\ref{def:weight-zero-HS} that we have equipped $\widetilde\rH(X,\bQ)$ and $\rH^\ev(X,\bQ)$ with Hodge structures of weight $0$. Similarly, we equip the odd cohomology of $X$ with a Hodge structure of weight $1$ as follows
\[
	\rH^\odd(X,\bQ) :=  \bigoplus_{i} \rH^{2i+1}(X,\bQ(i)).
\]

\begin{lemma}\label{lemma:respects-hodge}
Let $g\in \GSpin(\widetilde\rH(X,\bQ))$. If the action of $g$ on $\widetilde\rH(X,\bQ)$ respects the Hodge structure, then so does its action on $\rH^\ev(X,\bQ)$ and on $\rH^\odd(X,\bQ)$.
\end{lemma}

\begin{proof}
This follows  immediately from the fact that the Hodge structure is determined by the action of $h'\in\fg(X)\otimes_\bQ\bC$ (see \S~\ref{subsec:bi-graded-lie}), and from the faithfulness of the $\fg(X)$-module $\widetilde\rH(X,\bQ)$. 
\end{proof}

\begin{theorem}Let $X_1$ and $X_2$ be hyperk\"ahler varieties, and let $\Phi\colon\cD X_1\isomto \cD X_2$ be an equivalence. Then for every $i$ the $\bQ$-Hodge structures $\rH^i(X_1,\bQ)$ and 
$\rH^i(X_2,\bQ)$ are isomorphic.
\end{theorem}

\begin{proof}
Consider the Lie algebra isomorphism $\Phi^\fg\colon \fg(X_1)\isomto \fg(X_2)$ from Theorem~\ref{bigthm:Lie-invariant}. By  Proposition~\ref{prop:similitude}, there exists a Hodge similitude $\phi\colon \widetilde\rH(X_1,\bQ) \isomto \widetilde\rH(X_2,\bQ)$ 
so that the square
\[
\begin{tikzcd}
\fso(\widetilde\rH(X_1,\bQ)) \arrow{r}{\ad(\phi)} \arrow{d} & \fso(\widetilde\rH(X_2,\bQ)) \arrow{d} \\
\fg(X_1)  \arrow{r}{\Phi^\fg} & \fg(X_2)
\end{tikzcd}
\]
commutes. Here the vertical maps are the isomorphisms from Theorem~\ref{thm:LLV}.

The K3-type Hodge structure $\widetilde\rH(X_2,\bQ)$ decomposes as $N\oplus T$, with $N$ and $T$ its algebraic and transcendental parts, respectively. The Hodge similitude $\phi$ maps the distinguished elements $\alpha_1$ and $\beta_1$ of $\widetilde\rH(X_1,\bQ)$ to ${N}$. By Witt cancellation, there exists a $\psi_N \in \SO(N)$ and $\lambda,\mu \in \bQ^\times$ such that $\psi_N\phi(\alpha_1) =  \lambda \alpha_2$ and $\psi_N\phi(\beta_1) = \mu \beta_2$. Extending by the identity, we find a Hodge isometry $\psi\in \SO(\widetilde\rH(X_2,\bQ))$ such that $\psi\phi\colon \widetilde\rH(X_1,\bQ)\isomto \widetilde\rH(X_2,\bQ)$ is a \emph{graded} Hodge similitude. In particular, the induced map $\psi\phi \colon \fg(X_1) \isomto \fg(X_2)$ is graded, and $\psi\phi$ maps the grading element $h_1\in \fg(X_1)$ to the grading element $h_2\in \fg(X_2)$.  

By the exact sequence (\ref{eq:gspin-so}) and by Hilbert~90, the element $\psi$ lifts to an element $\widetilde\psi \in \GSpin(\widetilde\rH(X_2,\bQ))$. By Lemma~\ref{lemma:gspin-acts} and Lemma~\ref{lemma:respects-hodge}, it induces automorphisms of the Hodge structures
$ \rH^\ev(X_2,\bQ)$ and $\rH^\odd(X_2,\bQ)$. Now by construction, the composition $\widetilde\psi \circ\Phi^\rH$ defines isomorphisms 
\begin{gather*}
	\widetilde\psi \circ \Phi^\rH\colon \rH^\ev(X_1,\bQ) \isomto \rH^\ev(X_2,\bQ) \\
	\widetilde\psi \circ\Phi^\rH\colon \rH^\odd(X_1,\bQ) \isomto \rH^\odd(X_2,\bQ) 
\end{gather*}
which respect both the grading and the Hodge structure, so that they induce isomorphisms of Hodge structures $\rH^i(X_1,\bQ) \isomto \rH^i(X_2,\bQ)$, for all $i$.
\end{proof}

\section{Topological $K$-theory}

\subsection{Topological $K$-theory and the Mukai vector}

In this paragraph we briefly recall some basic properties of topological $K$-theory of projective algebraic varieties. See \cite{AH61,AH62,AddingtonThomas14} for more details.

For every smooth and projective $X$ over $\bC$ we have a $\bZ/2\bZ$-graded abelian group
\[
	\rK_\top(X) := \rK_\top^0(X) \oplus \rK_\top^1(X).
\]
This is functorial for pull-back and proper pushforward, and carries a product structure.
The group $\rK^0_{\top}(X)$ is  the Grothendieck group of topological vector bundles on the differentiable manifold $X^\an$. Pull-back agrees with pull-back of vector bundles, and the product structure agrees with the tensor product of vector bundles. There is a natural $\bZ/2\bZ$-graded map
\[
	v_X^\top\colon \rK_\top(X) \to \rH(X,\bQ),
\]
which in even degree is given by $v_X^\top(\cF) = \sqrt{\Td_X} \cdot \ch(\cF)$. The image of $v_X^\top$ is a $\bZ$-lattice of full rank.

There is a `forgetful map' $\rK^0(X) \to \rK_\top(X)$ from the Grothendieck group of algebraic vector bundles (or equivalently of the triangulated category $\cD X$). This is compatible with pull back, multiplication, and proper pushfoward.  The mukai vector $v_X\colon \rK^0(X) \to \rH(X,\bQ)$ factors over $v_X^\top$.

If $\cP$ is an object in $\cD(X\times Y)$ then convolution with its class in $\rK_\top^0(X\times Y)$ defines a map $\Phi^\rK_\cP \colon \rK_\top(X) \to \rK_\top(Y)$, in such a way that the diagram
\[
\begin{tikzcd}
	\rK^0(X) \arrow{r}\arrow{d}{\Phi_\cP} & \rK_{\top}(X) \arrow{d}{\Phi_\cP^\rK} \arrow{r}{v_X^\top}
	& \rH(X,\bQ) \arrow{d}{\Phi_\cP^\rH} \\ 
	\rK^0(Y) \arrow{r} & \rK_{\top}(Y)\arrow{r}{v_Y^\top} & \rH(Y,\bQ).
\end{tikzcd}
\]
commutes.

\subsection{Equivariant topological $K$-theory}\label{subsec:equivariant-topological-K-theory}
The above formalism largely generalizes to an equivariant setting. Again we briefly recall the most important properties, see \cite{Segal68,AtiyahSegal68,BFM79,Joshua99} for more details. 

If $X$ is a smooth projective complex variety equipped with an action of a finite group $G$ we denote by $\rK^0_G(X)$  the Grothendieck group of $G$-equivariant algebraic vector bundles on $X$, or equivalently the Grothendieck group of  the bounded derived category $\cD_G X$ of $G$-equivariant coherent $\cO_X$-modules. This  is functorial for pull-back along $G$-equivariant maps and push-forward along $G$-equivariant proper maps.

Similarly, we have the $G$-equivariant topological $K$-theory   
\[
	\rK_{\top,G}(X) := \rK_{\top,G}^0(X) \oplus \rK_{\top,G}^1(X),
\]
where $\rK_{\top,G}^0(X)$ is the Grothendieck group of topological $G$-equivariant vector bundles. 

There is a natural map $\rK^0_G(X)\to \rK^0_{\top,G}(X)$ compatible with pull-back and tensor product. If $f\colon X\to Y$ is \emph{proper} and $G$-equivariant, then we have a push-forward map $f_\ast \colon \rK_{\top,G}(X) \to \rK_{\top,G}(Y)$.  There is a Riemann--Roch theorem \cite{AtiyahSegal68,Joshua99}, stating that the square
\[
\begin{tikzcd}
\rK^0_G(X)  \arrow{d}{f_\ast} \arrow{r} & \rK_{\top,G}(X) \arrow{d}{f_\ast} \\
\rK^0_G(Y) \arrow{r} & \rK_{\top,G}(Y)
\end{tikzcd}
\]
commutes.

Now assume that we have a finite group $G$ acting on $X$, and a finite group $H$ acting on $Y$. If
$\cP$ is an object in $\cD_{G\times H}(X\times Y)$, then convolution with $\cP$  induces a functor
$\Phi_\cP \colon \cD_G X \to \cD_H Y$, see \cite{Ploog07} for more details. Similarly, convolution with
the class of $\cP$ in $\rK^0_{\top,G\times H}(X\times Y)$ induces a map $\Phi_\cP^\rK\colon\rK_{\top,G}(X) \to \rK_{\top,H}(Y)$. These satisfy the usual Fourier--Mukai calculus, and moreover they are compatible, in the sense that the square
\[
\begin{tikzcd}
	\rK^0_G(X) \arrow{r}\arrow{d}{\Phi_\cP} & \rK_{\top,G}(X) \arrow{d}{\Phi_\cP^\rK} \\
	\rK^0_H(Y) \arrow{r} & \rK_{\top,H}(Y).
\end{tikzcd}
\]
commutes.

\section{Cohomology of the Hilbert square of a K3 surface}

Let $S$ be a K3 surface and $X=S^{[2]}$ its Hilbert square. In the coming few paragraphs we recall the structure of the cohomology of $X$ in terms of the cohomology of $S$. See \cite{Beauville83,EGL,Huybrechts03} for more details.

\subsection{Line bundles on the Hilbert square}\label{subsec:linebundles}

Let $G=\{1,\sigma\}$ be the group of order two, acting on $S\times S$ by permuting the factors.
The Hilbert square $X$ sits in a diagram 
\[
\begin{tikzcd}
& Z \arrow{rd}{q} \arrow[swap]{ld}{p} &  \\
S\times S & & X
\end{tikzcd}
\]
where $p \colon Z\to S\times S$ is the blow-up along the diagonal, and where $q\colon Z\to X$ is the quotient map for the natural action of $G$ on $Z$. Denote by $R \subset Z$ the exceptional divisor of $p$. Then $R$ equals the ramification locus of $q$. We have $q_\ast \cO_Z = \cO_X \oplus \cE$ for some line bundle $\cE$, and $q^\ast \cE \cong \cO_Z(-R)$.

If $\cL$ is a line bundle on $S$ then 
\[
	\cL_2 := \big( q_\ast p^\ast (\cL\boxtimes \cL) \big)^G %= {p'}^\ast \big( q'_\ast (L \boxtimes L) \big)^G
\]
is a line bundle on $X$.  The map
\[
	\Pic(S) \oplus \bZ \longto \Pic(X)\colon  (\cL,n) \mapsto \cL_2 \otimes \cE^{\otimes n}
\]
is an isomorphism. %This follows for e from the fact that any topological line bundle becomes algebraic after a suitable
%deformation of the complex structure on $S$. 

\subsection{Cohomology of the Hilbert square}\label{subsec:rational-cohomology}
There is an isomorphism
\[
	\rH^2(S,\bZ) \oplus \bZ\delta \isomto \rH^2(X,\bZ)
\]
with the property that $c_1(\cL)$ is mapped to $c_1(\cL_2)$, and $\delta$ is mapped to $c_1(\cE)$. We will use this isomorphism to identify $\rH^2(S,\bZ) \oplus \bZ\delta$ with $\rH^2(X,\bZ)$.
The Beauville--Bogomolov form on $\rH^2(X,\bZ)$ satisfies
\[
	b_X(\lambda,\lambda) = b_S(\lambda,\lambda),\quad b_X(\lambda,\delta)=0,\quad b_X(\delta,\delta)=-2
\]
for all $\lambda \in \rH^2(S,\bZ)$.

Cup product defines an isomorphism $\Sym^2 \rH^2(X,\bQ) \isomto \rH^4(X,\bQ)$. By Poincar\'e duality, there is a unique $q_X \in \rH^4(X,\bQ)$ representing the Beauville--Bogomolov form, in the sense that
\begin{equation}\label{eq:class-q}
	\int_X q_X \lambda_1\lambda_2 = b_X(\lambda_1,\lambda_2)
\end{equation}
for all $\lambda_1,\lambda_2 \in \rH^2(X,\bZ)$. Multiplication by $q_X$ defines an isomorphism $\rH^2(X,\bQ) \to \rH^6(X,\bQ)$ and for all
$\lambda_1,\lambda_2,\lambda_3$ in $\rH^2(X,\bQ)$ 
we have
\begin{equation}\label{eq:q-lambda}
	\lambda_1\lambda_2\lambda_3 = 
		b_X(\lambda_1,\lambda_2) q_X \lambda_3 + 
		b_X(\lambda_2,\lambda_3) q_X \lambda_1 + 
		b_X(\lambda_3,\lambda_1) q_X \lambda_2
\end{equation}
in $\rH^6(X,\bQ)$. Finally, for all $\lambda \in \rH^2(X,\bQ)$ the Fujiki relation
\begin{equation}\label{eq:fujiki}
	\int_X \lambda^4 = 3 b_X(\lambda,\lambda)^2
\end{equation}
holds.

\subsection{Todd class of the Hilbert square}

\begin{proposition}\label{prop:Todd}
$\Td_X = 1 + \frac{5}{2} q_X +  3[\pt]$.
\end{proposition}

\begin{proof}See also \cite[\S 23.4]{Huybrechts03}.
Since the Todd class is invariant under the monodromy group of $X$, we necessarily have
\[
	\Td_X = 1 + s q_X + t [\pt]
\]
for some $s,t\in \bQ$. By Hirzebruch--Riemann--Roch, for every line bundle $L$ on $S$ with $c_1(L)=\lambda$ we have
\[
	\chi(X,L_2) = \int_X \ch(\lambda) \Td_X = 
	\frac{1}{24} \int_X {\lambda^4} + \frac{s}{2} \int_X  \lambda^2 q_X + t.
\]
By the relations (\ref{eq:fujiki}) and (\ref{eq:class-q}) the right hand side reduces to 
\[
	\frac{b(\lambda,\lambda)^2}{8} + \frac{s b(\lambda,\lambda)}{2} + t.
\]
By \cite[\S 23.4]{Huybrechts03} or \cite[5.1]{EGL} the left hand side computes to
\[
	\chi(X,L_2) = \frac{b(\lambda,\lambda)^2}{8} + \frac{5}{4}b(\lambda,\lambda) + 3.
\]
Comparing the two expressions yields the result.
\end{proof}

\section{Derived McKay correspondence}

\subsection{The derived McKay correspondence}
%We briefly recall the derived McKay correspondence due to Bridgeland--King--Reid \cite{BKR} and Haiman \cite{Haiman}. 
As in \S~\ref{subsec:linebundles}, we consider a K3 surface $S$, its Hilbert square $X=S^{[2]}$,  the maps 
$p\colon Z\to S\times S$ and $q\colon Z\to X$, and the group  $G=\{1,\sigma\}$ acting on $S\times S$ and $Z$. 

The \emph{derived McKay correspondence} \cite{BKR} is the triangulated functor $\BKR\colon \cD^b(X) \to \cD^b_G(S\times S)$ given as the composition
\[
	\BKR\colon \cD X \overset{q^\ast}\longto  \cD_G(Z) \overset{p_\ast}\longto \cD_G(S\times S),
\]
where the first functor maps $\cF$ to $q^\ast \cF$ equipped with the trivial $G$-linearization. By \cite[1.1]{BKR} and \cite[5.1]{Haiman} the functor $\BKR$ is an equivalence of categories. Its inverse is given by:
\begin{proposition}\label{prop:BKR-inverse}
The inverse equivalence of $\BKR$ is given by
\[
	\BKR^{-1}(\cF) =  (q_\ast p^\ast \cF)^{\sigma = -1} \otimes \cE^{-1}
\]
for all $\cF$ in $\cD_G(S\times S)$. 
\end{proposition}

\begin{proof}
This follows from combining \cite[4.1]{Krug15} with the projection formula for $q$ and the fact that $q^\ast \cE \cong \cO_Z(-R)$.
\end{proof}

Now let $S_1$ and $S_2$ be K3 surfaces with Hilbert squares $X_1$ and $X_2$. As was observed by Ploog \cite{Ploog05}, any  equivalence $\Phi\colon \cD S_1 \isomto \cD S_2$  induces an equivalence $\cD_G(S_1\times S_2) \isomto \cD_G(S_2\times S_2)$, and hence via  the derived McKay correspondence an equivalence $\Phi^{[2]}\colon \cD X_1 \isomto \cD X_2$.

\subsection{Topological $K$-theory of the Hilbert square}

\begin{theorem}\label{thm:topological-BKR}
The composition
\[
	\BKR_\top\colon 
	\rK_\top(X) \overset{q^\ast}\longto  \rK_{\top,G}(Z) \overset{p_\ast}\longto \rK_{\top,G}(S\times S)
\]
is an isomorphism.
\end{theorem}	

\begin{proof}
(See also \cite[\S~10]{BKR}). 
This is a purely formal consequence of the calculus of equivariant Fourier--Mukai transforms sketched in \S~\ref{subsec:equivariant-topological-K-theory}. The functor $\BKR$ and its inverse are given by kernels $\cP \in \cD_G(X\times S\times S)$ and $\cQ\in \cD_G(S\times S\times X)$. The map $\BKR_\top$ is given by convolution with the class of $\cP$ in $\rK^0_{\top,G}(X\times S\times S)$. The identities in $\rK^0(X\times X)$ and $\rK^0_{G\times G}(S\times S\times S\times S)$ witnessing that $\cP$ and $\cQ$ are mutually inverse equivalences  induce analogous identities in $\rK^0_\top$. These show that convolution with the class of $\cQ$ defines a two-sided inverse to $\BKR_\top$. 
\end{proof}

Consider the map
\[	
	\psi^\rK\colon \rK^0_\top(X) \to \rK^0_{\top}(S\times S)^G 
\]
obtained as the composition of $\BKR_\top$ and the forgetful map from $\rK^0_{\top,G}(S\times S)$ to $\rK^0_\top(S\times S)$. Also, consider the map
\[
	\theta^\rK\colon \rK^0_\top(S) \to \rK^0_\top(X),\, 
	[\cF] \mapsto \BKR^{-1}_\top\big([\cF \boxtimes \cF, 1] - [\cF \boxtimes \cF, -1]\big),
\]
where $[\cF\boxtimes \cF,\pm 1]$ denotes the class of the topological vector bundle $\cF \boxtimes \cF$ equipped with $\pm$  the natural $G$-linearisation. 

By construction, these maps are `functorial' in $\cD S$, in the following sense:

\begin{proposition}\label{prop:theta-equivariance}
 If $\Phi\colon \cD S_1 \isomto \cD S_2$ is a derived equivalence between K3 surfaces, and  $\Phi^{[2]}\colon \cD X_1 \isomto \cD X_2$ the induced equivalence between their Hilbert squares, then the squares 
\[
\begin{tikzcd}
\rK^0_\top(X_1) \arrow{r}{\psi^\rK} \arrow{d}{\Phi^{[2],\rK}} 
	& \rK^0_{\top}(S_1\times S_1)^G  \arrow{d}{\Phi^\rK\otimes\Phi^\rK}
&& \rK^0_\top(S_1) \arrow{r}{\theta^\rK} \arrow{d}{\Phi^\rK} 
	& \rK^0_\top(X_1)  \arrow{d}{\Phi^{[2],\rK}} \\
\rK^0_\top(X_2) \arrow{r}{\psi^\rK}  & \rK^0_{\top}(S_2\times S_2)^G
&& \rK^0_\top(S_2) \arrow{r}{\theta^\rK} & \rK^0_\top(X_2) 
\end{tikzcd}
\] 
commute.\qed
\end{proposition}

\begin{proposition}\label{prop:topological-K-theory}
The sequence
\[
	0 \longto  \rK^0_\top(S) \otimes_\bZ\bQ
	\overset{\theta^\rK}\longto 
	 \rK^0_{\top}(X) \otimes_\bZ\bQ
	   \overset{\psi^\rK}\longto \rK^0_{\top}(S\times S)^G \otimes_\bZ\bQ
	\longto 0
\]
is exact.
\end{proposition}

\begin{proof}
In the proof, we will  implicitly identify $\rK_{\top,G}(S\times S)$ and $\rK_\top(X)$.

Note that the map $\theta^\rK$ is additive. Indeed, let $\cF_1$ and $\cF_2$ be (topological) vector bundles on $S$.  Then the cross term
$\theta^{\rK}[\cF_1\oplus \cF_2] - \theta^{\rK}[\cF_1] - \theta^{\rK}[\cF_2]$ computes to

\[
	\Big[ \cF_1\boxtimes \cF_2 \oplus \cF_2 \boxtimes \cF_1, \,
		\left(\begin{smallmatrix} 0 & 1 \\1 & 0\end{smallmatrix}\right) \Big] - 
	\Big[ \cF_1\boxtimes \cF_2 \oplus \cF_2 \boxtimes \cF_1, \,
		\left(\begin{smallmatrix} 0 & -1 \\-1 & 0\end{smallmatrix}\right) \Big],
\]
which vanishes because the matrices $\left(\begin{smallmatrix} 0 & 1 \\1 & 0\end{smallmatrix}\right)$
and $\left(\begin{smallmatrix} 0 & -1 \\-1 & 0\end{smallmatrix}\right)$ are conjugated over $\bZ$. 

The composition $\psi^\rK \theta^\rK$ vanishes, and since the Schur multiplier of $G$ is trivial, the map $\psi^\rK$ is surjective. Computing the $\bQ$-dimensions one sees that it suffices to show that $\theta^\rK$ is injective. 

Pulling back to the diagonal and taking invariants defines a map
\[
	\rK^0_\top(S) \overset{\theta^\rK}{\longto} \rK^0_{\top,G}(S\times S)
	\overset{\Delta^\ast}\longto \rK^0_{\top,G}(S) \overset{(-)^G}\longto \rK^0_\top(S).
\]
This composition computes to
\[
	[\cF] \mapsto \ [\Sym^2 \cF] - [\wedge^2 \cF].
\] 
This coincides with the second Adams operation, which is injective on  $\rK^0_\top(S)\otimes_\bZ \bQ$, 
since it has eigenvalues $1$, $2$, and $4$. We conclude that $\theta^\rK$ is injective, and the proposition follows.
\end{proof}

\begin{remark}One can show that the sequence
\[
	0 \longto  \rK^0_\top(S) 
	\overset{\theta^\rK}\longto 
	 \rK^0_{\top}(X) 
	   \overset{\psi^\rK}\longto \rK^0_{\top}(S\times S)^G 
	\longto 0
\]
with integral coefficients is exact.
\end{remark}

\subsection{A computation in the cohomology of the Hilbert square}\label{subsec:theta}
This paragraph contains the technical heart of our computation of the derived monodromy of the Hilbert square of a K3 surface.

Consider the map $\theta^\rH\colon \rH(S,\bQ) \to \rH(X,\bQ)$ given by
\begin{equation}\label{eq:explicit-theta}
	\theta^\rH(s+\lambda + t\pt_S) = \big( s\delta + \lambda \delta + t q_X\delta \big) \cdot e^{-\delta/2},
\end{equation}
for all $s,t \in \bQ$ and $\lambda \in \rH^2(S,\bQ)$. See \S~\ref{subsec:rational-cohomology} for the definition of  $\delta \in \rH^2(X,\bQ)$ and $q_X \in \rH^4(X,\bQ)$.

\begin{proposition}\label{prop:explicit-theta}
The square
\[
\begin{tikzcd}
	\rK^0_\top(S) \arrow{r}{\theta^\rK} \arrow{d}{v_S^\top} & \rK^0_\top(X) \arrow{d}{v_X^\top} \\
	\rH(S,\bQ) \arrow{r}{\theta^\rH} & \rH(X,\bQ)
\end{tikzcd}
\]
commutes.
\end{proposition}

\begin{proof}
Since $\rK^0_\top(S) \otimes_\bZ \bQ$ is additively generated by line bundles, it suffices to show 
\begin{equation}\label{eq:technical-heart}
	v_X^\top(\theta^\rK(\cL)) =
	 \Big( \delta + \lambda \delta + (\tfrac{b(\lambda,\lambda)}{2} + 1) q_X \delta \Big)
	  \cdot e^{-\delta/2}
\end{equation}
for a topological line bundle $\cL$ with $\lambda=c_1(\cL)$. Deforming $S$ if necessary, we may assume that $\cL$ is algebraic.

Using Proposition~\ref{prop:BKR-inverse} and the fact that the natural map 
\[
	\cL_2 \otimes q_\ast \cO_Z \to q_\ast p^\ast(\cL \boxtimes \cL)
\]
 is an isomorphism of $\cO_X$-modules we find
\begin{align*}
	\BKR^{-1} [ \cL\boxtimes \cL, 1] &=  \cL_2 \\
	\BKR^{-1} [\cL\boxtimes \cL,-1] &= \cE^{-1} \otimes \cL_2.
\end{align*}
We conclude that $\theta^\rK$ maps  $\cL$ to $[\cL_2]( 1- [\cE^{-1}])$ in $\rK^0(X)$. 

We compute its image under  $v_X$. Using the formula for the Todd class from Proposition~\ref{prop:Todd} we find
\[
	v_X(\theta^\rK(\cL)) = (1 + \tfrac{5}{4} q_X + \cdots ) \exp(\lambda) (1-e^{-\delta}). 
\]
Since $1-e^{-\delta}$ has no term in degree $0$, the degree $8$ part of the square root of the Todd class is irrelevant, so we have
\[
	v_X(\theta^\rK(\cL)) = (1 + \tfrac{5}{4} q_X  ) \exp(\lambda) (1-e^{-\delta}). 
\]
By the Fujiki relation (\ref{eq:fujiki}) from \S~\ref{subsec:rational-cohomology}, we have
$\lambda^3\delta=0$, so the above can be rewritten as
\[
	v_X(\theta^\rK(\cL)) = (1 + \tfrac{5}{4} q_X  )\cdot
	\Big( \delta + \lambda \delta + \tfrac{\lambda^2}{2} \delta\Big) \cdot \frac{1-e^{-\delta}}{\delta}. 
\]
Since we have $q_X\delta\lambda=b(\delta,\lambda)=0$, we can rewrite this further as
\[
	v_X(\theta^\rK(\cL)) = (1 + \tfrac{1}{4} q_X) \cdot
	\Big( \delta + \lambda \delta + (\tfrac{b(\lambda,\lambda)}{2} + 1) q_X \delta \Big)
	\cdot \frac{1-e^{-\delta}}{\delta}.
\]
Comparing this with the right-hand-side of (\ref{eq:technical-heart}) we see that it suffices to show
\[
	(1 + \tfrac{1}{4} q_X) \cdot (1-e^{-\delta}) = \delta e^{-\delta/2}
\]
in $\rH(X,\bQ)$. This boils down to the identities
\[
	\frac{1}{6} \delta^3 + \frac{1}{4} \delta q_X = \frac{1}{8} \delta^3, \quad
	\frac{1}{24} \delta^4 + \frac{1}{8} \delta^2 q_X = \frac{1}{48} \delta^4
\]
in $\rH^6(X,\bQ)$ and $\rH^8(X,\bQ)$ respectively. These  follow easily from the relations (\ref{eq:class-q}), (\ref{eq:q-lambda}), and (\ref{eq:fujiki}) in \S~\ref{subsec:rational-cohomology}.
\end{proof}

\section{Derived monodromy group of the Hilbert square of a K3 surface}

\subsection{Derived monodromy groups}
Let $X$ be a smooth projective complex variety. We call a \emph{deformation} of $X$ the data of a smooth projective variety $X'$, a proper smooth family $\cX \to B$ , a path $\gamma\colon [0,1]\to \cX$ and isomorphisms $X\isomto \cX_{\gamma(0)}$ and $X'\isomto \cX_{\gamma(1)}$. We will informally say that $X'$ as a deformation of $X$, the other data being implicitly understood. Parallel transport along $\gamma$ defines an isomorphism $\rH(X,\bQ) \isomto \rH(X',\bQ)$.

If $X'$ and $X''$ are deformations of $X$, and if $\phi\colon X' \to X''$ is an isomorphism of projective varieties, then we obtain a composite isomorphism
\[
	\rH(X,\bQ) \longisomto \rH(X',\bQ) 
	\overset{\phi}\longisomto \rH(X'',\bQ) \longisomto 
	\rH(X,\bQ).
\]
We call such isomorphism a \emph{monodromy operator} for $X$, and denote by $\Mon(X)$ the subgroup of $\GL(\rH(X,\bQ))$ generated by all monodromy operators. 

If $X'$ and $X''$ are deformations of $X$, and if $\Phi\colon \cD X' \isomto \cD X''$ is an equivalence, then we obtain an isomorphism
\[
	\rH(X,\bQ) \longisomto \rH(X',\bQ) 
	\overset{\Phi^\rH}\longisomto \rH(X'',\bQ) \longisomto 
	\rH(X,\bQ).
\]
We call such isomorphism a \emph{derived monodromy operator} for $X$, and denote by $\DMon(X)$ the subgroup of $\GL(\rH(X,\bQ))$ generated by all derived monodromy operators. 

%\begin{remark}This is perhaps a poor man's version of a more refined 'derived monodromy group' that refers only to (a dg enhancement of) the category $\cD X$, and is defined in terms of `non-commutative' deformations.
%\end{remark}

By construction, the derived monodromy group is deformation invariant. It contains the usual monodromy group, and the image of $\rho_X$, and we have a commutative square of groups
\[
\begin{tikzcd}
\Aut(X) \arrow[hook]{r} \arrow{d} & \Aut(\cD X) \arrow{d}{\rho_X} \\
\Mon(X) \arrow[hook]{r} & \DMon(X). 
\end{tikzcd}
\]

\begin{remark}
The above definition is somewhat ad hoc, and should be considered a poor man's derived monodromy group. This is sufficient for our purposes. A more mature definition should involve all non-commutative deformations of $X$. 
\end{remark}

\begin{proposition}
If $S$ is a K3 surface, then $\DMon(S) = \rO^+(\widetilde\rH(S,\bZ))$.
\end{proposition}

\begin{proof}
Indeed, if $\Phi\colon \cD S_1 \to \cD S_2$ is an equivalence, then 
\[
	\Phi^\rH \colon \widetilde\rH(S_1,\bZ) \to \widetilde\rH(S_2,\bZ)
\]
preservers the Mukai form, as well as a natural orientation on four-dimensional positive subspaces (see \cite[\S~4.5]{HMS}). Also any deformation preserves the Mukai form and the natural orientation, so any derived monodromy operator will land in $\rO^+(\widetilde\rH(S,\bZ))$.

The converse inclusion can be easily obtained from the Torelli theorem, together with the results of \cite{Ploog05,HLOY} on derived auto-equivalences of K3 surfaces. Alternatively one can use that the group $\rO^+(\widetilde\rH(S,\bZ))$ is generated by relfections in $-2$ vectors $\delta$.  By the Torelli theorem, any such $-2$-vector will become algebraic on a suitable deformation $S'$ of $S$, and by \cite{Kuleshov90} there exists a spherical object $\cE$ on $S'$ with Mukai vector $v(\cE)=\delta$. The spherical twist in $\cE$ then shows that reflection in $\delta$ is indeed a derived monodromy operator.
\end{proof}

\subsection{Action of $\DMon(S)$ on $\rH(X,\bQ)$}

By the derived McKay correspondence, any derived equivalence $\Phi_S \colon \cD S_1 \isomto \cD S_2$ between K3 surfaces induces a derived equivalence $\Phi_X \colon \cD X_1 \isomto \cD X_2$ between the corresponding Hilbert squares. By Propositions~\ref{prop:theta-equivariance} and~\ref{prop:topological-K-theory}, the induced map $\Phi_X^H$ only depends on $\Phi_S^H$. Since  any deformation of a K3 surface~$S$ induces a deformation of $X=S^{[2]}$, we conclude that we have a natural homomorphism
\[
	\DMon(S) \longto \DMon(X),
\]
and hence an action of $\DMon(S)$ on $\rH(X,\bQ)$. In this paragraph, we will explicitly compute this action. As a first approximation, we determine the $\DMon(S)$-module structure of $\rH(X,\bQ)$, up to isomorphism.

\begin{proposition}\label{prop:abstract-module-structure}
We have $\rH(X,\bQ) \cong  \widetilde{\rH}(S,\bQ) \oplus \Sym^2 \widetilde{\rH}(S,\bQ)$ as representations
of $\DMon(S)=\rO^+(\widetilde{\rH}(S,\bZ))$.
\end{proposition}

\begin{proof}
This follows from Propositions~\ref{prop:theta-equivariance} and~\ref{prop:topological-K-theory}.
\end{proof}

Since $\fg(X)$ is a purely topological invariant, it is preserved under deformations. In particular,  Theorem \ref{thm:normalizer} implies that we have an inclusion $\DMon(X) \subset \rO(\widetilde\rH(X,\bQ))$.
We conclude there exists a unique map of algebraic groups
$h$ making the square
\begin{equation}\label{eq:BKR-action}
\begin{tikzcd}
\DMon(S) \arrow{r} \arrow{d} & \DMon(X)\arrow{d} \\
\rO(\widetilde\rH(S,\bQ)) \arrow{r}{h} & \rO(\widetilde\rH(X,\bQ))
\end{tikzcd}
\end{equation}
commute. 

\begin{theorem}\label{thm:BKR-action}
The map $h$ in the square (\ref{eq:BKR-action}) is given by
\[
	 g \mapsto \det(g) \cdot \big(B_{-\delta/2} \circ \iota(g) \circ B_{\delta/2} \big),
\]
with $\iota\colon \rO(\widetilde\rH(S,\bQ)) \to \rO(\widetilde\rH(X,\bQ))$ the natural inclusion. 
\end{theorem}

The proof of this theorem will occupy the remainder of this paragraph.

\medskip

%Recall from Section~\ref{sec:LL-so} that we have Lie algebras $\fg(S) = \fso(\widetilde\rH(S,\bQ))$ and $\fg(X) = \fso(\widetilde\rH(X,\bQ))$ acting on 
%$\rH(S,\bQ)$ and $\rH(X,\bQ)$ respectively.
Consider the unique homomorphism of  Lie algebras $\iota\colon \fg(S) \to \fg(X)$ that respects the grading and maps $e_\lambda$ to $e_\lambda$ for all $\lambda \in \rH^2(S,\bQ) \subset \rH^2(X,\bQ)$. Under the isomorphism of Theorem~\ref{thm:LLV} 
this corresponds to the map $\fso(\widetilde\rH(S,\bQ)) \to \fso(\widetilde\rH(X,\bQ))$ induced by the inclusion of quadratic spaces $\widetilde\rH(S,\bQ) \subset \widetilde\rH(X,\bQ)$.

Recall from Section~\ref{subsec:theta} the map $\theta^\rH\colon \rH(S,\bQ) \to \rH(X,\bQ)$.

\begin{lemma}\label{lemma:explicit-lie-equivariance}
The map $\theta^\rH \colon \rH(S,\bQ) \to \rH(X,\bQ)$ is equivariant with respect to 
\[
	\theta^\fg \colon \fg(S) \to \fg(X),\, x \mapsto B_{-\delta/2} \circ \iota(x) \circ B_{\delta/2}.
\]
\end{lemma}

\begin{proof}
We have $\theta^\rH = e^{-\delta/2} \cdot \theta^\rH_0$, with
\[
	\theta^\rH_0(s+\lambda + t\pt_S) =  s\delta + \lambda \delta + t q_X\delta.
\]
The map $\theta^\rH_0$ respects the grading, and we claim that for every $\mu \in \rH^2(S,\bQ)$ the diagram
\[
 \begin{tikzcd}
 \rH(S,\bQ) \arrow{r}{\theta^\rH_0} \arrow{d}{e_\mu} 
 	& \rH(X,\bQ) \arrow{d}{e_\mu} \arrow{r}{e^{-\delta/2}} 
	& \rH(X,\bQ) \arrow{d}{e^{-\delta/2} e_\mu e^{\delta/2}} \\
 \rH(S,\bQ) \arrow{r}{\theta^\rH_0} 
 	& \rH(X,\bQ) \arrow{r}{e^{-\delta/2}}
	& \rH(X,\bQ)
 \end{tikzcd}
 \]
 commutes. Indeed, we have
 \begin{align*}
 	e_\mu(\theta^\rH_0(s+\lambda + t\pt_S)) &= s \delta \mu + \lambda \delta  \mu + t q_X  \delta  \mu, \\
	\theta^\rH_0(e_\mu(s+\lambda + t\pt_S)) &= s \delta  \mu + b(\lambda,\mu) q_X  \delta.
\end{align*}
One verifies easily that these agree, using identities (\ref{eq:q-lambda}) and (\ref{eq:class-q}) from \S~\ref{subsec:rational-cohomology}, and the fact that $b(\lambda,\delta)=b(\mu,\delta)=0$. This shows that the left-hand square commutes. The right-hand square commutes trivially, so the outer rectangle commutes, which shows that $\theta^\rH = e^{-\delta/2} \cdot \theta^\rH_0$ is indeed equivariant with respect to $\theta^\fg$.
\end{proof}

\begin{lemma}\label{lemma:complement}
There is an isomorphism 
\[
	\det (\widetilde\rH(X,\bQ)) \otimes \Sym^2 (\widetilde\rH(X,\bQ)) \cong \rH(X,\bQ) \oplus \det(\widetilde\rH(X,\bQ))
\]
of representations of $G=\rO(\widetilde\rH(X,\bQ))$.
\end{lemma}

\begin{proof}
This follows from  Lemma~\ref{lemma:exact-sequence}, Theorem~\ref{thm:normalizer} and Remark~\ref{rmk:explicit-normalizer}. 
\end{proof}

We are now ready to prove the main result of this paragraph.

\begin{proof}[Proof of Theorem~\ref{thm:BKR-action}]
By Proposition~\ref{prop:explicit-theta} the map $\theta^\rH$ is equivariant for the action of $\DMon(S)$. Lemma~\ref{lemma:explicit-lie-equivariance} then implies that 
\[
	h(g) =  B_{-\delta/2} \circ \iota(g) \circ B_{\delta/2} 
\]
for all $g\in \SO(\widetilde\rH(S,\bQ))$. We have an orthogonal decomposition
\[
	\widetilde\rH(X,\bQ) = B_{-\delta/2} (\widetilde\rH(S,\bQ)) \oplus  C
\]
with $C$ of rank $1$. Since $\SO(\widetilde\rH(S,\bQ))$ is normal in $\rO(\widetilde\rH(S,\bQ))$, the action of $\rO(\widetilde\rH(S,\bQ))$ (via $h$) must preserve this decomposition. With respect to this decomposition $h$ must then be given by
\[
	h(g) = \big(B_{-\delta/2} \circ g \epsilon_1(g) \circ B_{\delta/2} \big)  \oplus \epsilon_2(g)
\] 
where the $\epsilon_i(g) \colon \rO(\widetilde\rH(S,\bQ)) \to \{\pm 1\}$ are quadratic characters. This leaves four possibilities for $h$. One verifies that $\epsilon_1=\epsilon_2=\det g$ is the only possibility compatible with Proposition~\ref{prop:abstract-module-structure} and Lemma~\ref{lemma:complement},  and the theorem follows.
\end{proof}

\subsection{A transitivity lemma}

In this section we prove a lattice-theoretical lemma that will play an important role in the proofs of Theorem~\ref{bigthm:derived-monodromy} and Theorem~\ref{bigthm:lower-bound}.

%We denote by $\rO^+(L)\subset \rO(L)$ and $\SO^+(L) \subset \SO(L)$ the subgroups consisting of those transformations that preserve the orientation of a maximal positive definite subspace of $L\otimes \bR$. If one normalizes the {real spinor norm} in such a way that reflection in a positive (resp.~negative) vector has spinor norm $-1$ (resp.~$1$), then these subgroups coincide with the spinor norm kernels.

Let $b\colon L\times L \to \bZ$ be an even non-degenerate lattice.
%  We denote the \emph{discriminant module} of $L$ by $D(L)=L^\vee/L$. For a non-zero $\lambda \in L$ we denote by $\divi \lambda$ the positive generator of $b(\lambda,L)$, and by  $\lambda^\ast$ the class of $\lambda/\divi \lambda$ in $D(L)$.
Let $U$ be a hyperbolic plane with basis consisting of isotropic vectors $\alpha$, $\beta$ satisfying $b(\alpha, \beta) = -1$. 

As before, to a $\lambda \in L$ we associate the isometry $B_\lambda \in \rO(U\oplus L)$  defined as
\[
	B_\lambda( r\alpha + \mu + s\beta) = 
	r\alpha + (\mu + r \lambda) + \big(s + b(\mu,\lambda) + r \tfrac{b(\lambda,\lambda)}{2}\big)\beta
\]
for all $r,s\in \bZ$ and $\mu \in L$.
Let $\gamma$ be the isometry of $U\oplus L$ given by $\gamma(\alpha)=\beta$, $\gamma(\beta)=\alpha$, and $\gamma(\lambda)=-\lambda$ for all $\lambda\in L$. 

\begin{lemma}\label{lemma:transitivity}
Let $L$ be an even lattice containing a hyperbolic plane. Let $G \subset \rO(U\oplus L)$ be the subgroup generated by $\gamma$ and by the $B_\lambda$ for all $\lambda \in L$. Then for all $\delta \in U \oplus L$ with $\delta^2=-2$ and for all $g \in \rO(U\oplus L)$ there exists a $g'\in G$ such that $g'g$ fixes $\delta$.
\end{lemma}

\begin{proof}
This follows from classical results of Eichler. A convenient modern source is \cite[\S 3]{GHS}, whose notation we adopt. The isometry $B_\lambda$ coincides with the Eichler transvection $t(\beta,-\lambda)$. The conjugate $\gamma B_\lambda \gamma^{-1}$ is the Eicher transvection $t(\alpha,\lambda)$. Hence $G$ contains the subgroup $E_U(L) \subset \rO(U\oplus L)$ of unimodular transvections with respect to $U$. By \cite[Prop.~3.3]{GHS}, there exists a $g'\in E_U(L)$ mapping $g\delta$ to~$\delta$.
\end{proof}

\subsection{Proof of Theorem~\ref{bigthm:derived-monodromy}} \label{subsec:lower-bound}
Let $X$ be a hyperk\"ahler variety of type $\mathrm{K3}^{[2]}$. Let $\delta \in \rH^2(X,\bZ)$ be any class satisfying $\delta^2=-2$ and $b(\delta, \lambda) \in 2\bZ$ for all $\lambda \in \rH^2(X,\bZ)$. For example, if $X=S^{[2]}$, we may take $\delta=c_1(\cE)$ as in \S~\ref{subsec:rational-cohomology}. 
Consider the integral lattice
\[
	\Lambda := B_{\delta/2} \big( \bZ\alpha \oplus \rH^2(X,\bZ) \oplus \bZ\beta \big)
	\subset \widetilde\rH(X,\bQ).
\]
The sub-group $\Lambda \subset \widetilde\rH(X,\bQ)$ does not depend on the choice of $\delta$. In this section, we will prove Theorem~\ref{bigthm:derived-monodromy}. More precisely, we will show:

\begin{theorem}\label{thm:derived-monodromy}
$\rO^+(\Lambda) \subset \DMon(X) \subset \rO(\Lambda)$.
\end{theorem}

We start with the lower bound.

\begin{proposition}\label{prop:lower-bound-dmon}We have
$\rO^+(\Lambda) \subset \DMon(X)$ as subgroups of $\rO(\widetilde\rH(X,\bQ))$.
\end{proposition}

\begin{proof}
Since the derived monodromy group is invariant under deformation, we may assume without loss of generality
that $X=S^{[2]}$ for a K3 surface $S$ and $\delta=c_1(\cE)$ as in \S~\ref{subsec:rational-cohomology}.

The shift functor $[1]$ on $\cD X$ acts as $-1$ on $\rH(X,\bQ)$, which coincides with the action of $-1 \in \rO(\widetilde\rH(X,\bQ))$. In particular, $-1\in \rO^+(\Lambda)$ lies in $\DMon(X)$, so it suffices to show that $\SO^+(\Lambda)$ is contained in $\DMon(X)$.

Consider the isometry $\gamma\in \rO^+(\widetilde\rH(S,\bQ))$ given by $\gamma(\alpha)=-\beta$, $\gamma(\beta)=-\alpha$ and $\gamma(\lambda)=\lambda$ for all $\lambda \in \rH^2(S,\bQ)$. Then $\det(\gamma)=-1$ and by Theorem~\ref{thm:BKR-action} its image $h(\gamma)$ interchanges $B_{\delta/2}\alpha$ and $B_{\delta/2}\beta$ and acts by $-1$ on
$B_{\delta/2} \rH^2(X,\bZ)$. Since $\gamma$ lies in $\DMon(S) \subset \rO(\widetilde\rH(S,\bQ))$, we have that $h(\gamma)$
 lies in $\DMon(X) \subset \rO(\widetilde\rH(X,\bQ))$.
 
  Let $G\subset \rO(\widetilde\rH(X,\bQ))$ be the subgroup generated by $h(\gamma)$ and the isometries $B_\lambda$ for $\lambda \in \rH^2(X,\bZ)$. Clearly $G$ is contained in $\DMon(X)$. 
 
Let $g$ be an element of $\SO^+(\Lambda)$, and consider the image $gB_{\delta/2} \delta$ of $B_{\delta/2} \delta$. By Lemma~\ref{lemma:transitivity} there exists a $g'\in G\subset \DMon(X)$ so that $g'g$ fixes $B_{\delta/2} \delta$. But then $g'g$ acts on
\[
	(B_{\delta/2} \delta)^\perp = B_{\delta/2} \big( \bZ\alpha \oplus \rH^2(S,\bZ) \oplus \bZ\beta \big)
\]
with determinant $1$, and preserving the orientation of a maximal positive subspace. In particular, $g'g$ lies in the image of $\DMon(S) \to \DMon(X)$, and we conclude that $g$ lies in $\DMon(X)$.  
\end{proof}

The proof of the upper bound is now almost purely group-theoretical.

\begin{proposition}
$\SO(\Lambda)$ is the unique maximal arithmetic subgroup of $\SO(\Lambda\otimes_\bZ \bQ)$ containing $\SO^+(\Lambda)$.
\end{proposition}

\begin{proof}
More generally, this holds for any even lattice $\Lambda$ with the property that the
quadratic form $q(x)=b(x,x)/2$ on the $\bZ$-module $\Lambda$ is semi-regular \cite[\S~IV.3]{Knus}. 

For such $\Lambda$, the group schemes  $\bSpin(\Lambda)$ and $\bSO(\Lambda)$ are smooth over $\Spec \bZ$, see \emph{e.g.}~\cite{Ikai05}. In particular, for every prime $p$ the subgroups
$\Spin(\Lambda \otimes \bZ_p)$ and $\SO(\Lambda\otimes\bZ_p)$ of   $\Spin(\Lambda\otimes \bQ_p)$ resp.~$\SO(\Lambda\otimes\bQ_p)$ are maximal compact subgroups. It follows that the groups
\[
	\Spin(\Lambda) = \Spin(\Lambda\otimes \bQ) \bigcap \prod_p \Spin(\Lambda \otimes \bZ_p)
\]
and
\[
	\SO(\Lambda) = \SO(\Lambda\otimes \bQ) \bigcap \prod_p \SO(\Lambda \otimes \bZ_p)
\]
are maximal arithmetic subgroups of $\Spin(\Lambda\otimes \bQ)$ resp.~$\SO(\Lambda\otimes\bQ)$. 

The subgroup $\SO^+(\Lambda) \subset \SO(\Lambda)$ is the kernel of the spinor norm, and
the short exact sequence $1 \to \mu_2 \to \bSpin \to \bSO \to 1$ of fppf sheaves on $\Spec \bZ$ induces an exact sequence of groups
\[
	1 \longto \{ \pm1 \} \longto \Spin(\Lambda) \longto \SO^+(\Lambda) \longto 1.
\]
Let $\Gamma \subset \SO(\Lambda\otimes\bQ)$ be a maximal arithmetic subgroup containing $\SO^+(\Lambda)$. Let $\widetilde\Gamma$ be its inverse image in $\Spin(\Lambda\otimes \bQ)$, so that we have an exact sequence
\[
	1 \longto \{\pm 1\} \longto \widetilde\Gamma \longto \Gamma \longto \bQ^\times/2
\]
Since the group $\widetilde\Gamma$ is arithmetic and contains $\Spin(\Lambda)$, we have $\widetilde\Gamma = \Spin(\Lambda)$. Moreover, $\Gamma$ normalizes $\SO^+(\Lambda)=\ker(\Gamma\to\bQ^\times/2)$, and as the normalizer of an arithmetic subgroup of $\SO(\Lambda\otimes\bQ)$ is again arithmetic, we have that $\Gamma$ must equal the normalizer of $\SO^+(\Lambda)$. But then $\Gamma$ contains $\SO(\Lambda)$, and we conclude $\Gamma=\SO(\Lambda)$.
\end{proof}

\begin{corollary} $\DMon(X) \subset \rO(\Lambda)$.
\end{corollary}

\begin{proof}
$\DMon(X)$ preserves the integral lattice $\rK_\top(X)$ in the representation $\rH(X,\bQ)$ of $\rO(\widetilde\rH(X,\bQ))$, and hence is contained in an arithmetic subgroup of $\rO(\widetilde\rH(X,\bQ))=\SO(\widetilde\rH(X,\bQ)) \times \{\pm 1\}$. By Proposition~\ref{prop:lower-bound-dmon} it contains $\SO^+(\Lambda)\times \{\pm 1\}$, so we conclude from the preceding proposition that $\DMon(X)$ must be contained in $\rO(\Lambda)$.
\end{proof}

Together with Proposition~\ref{prop:lower-bound-dmon} this proves Theorem~\ref{thm:derived-monodromy}.

\section{The image of $\Aut(\cD X)$ on $\rH(X,\bQ)$}

\subsection{Upper bound for the image of $\rho_X$}\label{subsec:hodge}
We continue with the notation of the previous section. In particular, we denote by $X$ a hyperk\"ahler variety of type $\mathrm{K3}^{[2]}$, and by $\Lambda \subset \widetilde\rH(X,\bQ)$ the lattice defined in \S~\ref{subsec:lower-bound}. We equip $\widetilde\rH(X,\bQ)$ with the weight~$0$ Hodge structure
\[
	\widetilde\rH(X,\bQ) = \bQ \alpha \oplus \rH^2(X,\bQ(1)) \oplus \bQ \beta.
\]
We denote by $\Aut(\Lambda) \subset \rO(\Lambda)$ the group of isometries of $\Lambda$ that preserve this Hodge structure.

\begin{proposition}\label{prop:hodge}
$\im(\rho_X) \subset \Aut(\Lambda)$. 
\end{proposition}

\begin{proof}
By Theorem~\ref{thm:derived-monodromy} we have $\im(\rho_X) \subset \rO(\Lambda)$. The Hodge structure on
\[
	\rH(X,\bQ) = \oplus_{n=0}^4 \rH^{2n}(X,\bQ(n)),
\]
induces a Hodge structure on $\fg(X) \subset \End(\rH(X,\bQ))$, which agrees with the Hodge structure on $\fso(\widetilde\rH(X,\bQ))$ induces by the Hodge structure on $\widetilde\rH(X,\bQ)$. If $\Phi\colon \cD X\isomto \cD X$ is an equivalence, then $\Phi^\rH \colon \rH(X,\bQ) \isomto \rH(X,\bQ)$ and $\Phi^\fg \colon \fg(X) \isomto \fg(X)$ are isomorphisms of $\bQ$-Hodge structures, from which it follows that $\Phi^\rH$ must land in $\Aut(\Lambda) \subset \rO(\Lambda)$.
\end{proof}

\subsection{Lower bound for the image of $\rho_X$}\label{subsec:the-image}
We write $\Aut^+(\Lambda)$ for the index $2$ subgroup $\Aut(\Lambda) \cap \rO^+(\Lambda)$ of $\Aut(\Lambda)$.  

\begin{theorem}
Let $S$ be a K3 surface and let $X$ be the Hilbert square of $S$. Assume that $\NS(X)$ contains a hyperbolic plane. Then $\Aut^+(\Lambda) \subset \im \rho_X \subset \Aut(\Lambda)$.
\end{theorem}

\begin{proof}In view of Proposition~\ref{prop:hodge} we only need to show the lower bound.
The argument for this is entirely parallel to the proof of Proposition \ref{prop:lower-bound-dmon}.
Recall that we have
\[
	\Lambda = B_{\delta/2} \big( \bZ \alpha \oplus \rH^2(S,\bZ((1)) \oplus \bZ \delta \oplus \bZ \beta \big).
\]

The shift functor $[1]\in  \Aut(\cD X)$ maps to $-1\in \Aut^+(\Lambda)$, so it suffices to show that $\Aut^+(\Lambda) \cap \SO(\Lambda)$ is contained in $\im \rho_X$.

Let $\gamma_S \in \Aut(\cD S)$ be the composition of the spherical twist in $\cO_S$ with the shift $[1]$.
On the Mukai lattice $\widetilde\rH(S,\bZ) = \bZ \alpha \oplus \rH^2(X,\bZ(1)) \oplus \bZ \beta$ this equivalence maps $\alpha$ to $-\beta$ and $\beta$ to $-\alpha$ and is the identity on $\rH^2(S,\bZ)$. Under the derived McKay correspondence this induces an autoequivalence $\gamma_X \in \Aut \cD X$. By Theorem~\ref{thm:BKR-action}, the automorphism $\rho_X(\gamma_X) \in \Aut(\Lambda)$   interchanges $B_{\delta/2}\alpha$ and $B_{\delta/2}\beta$ and acts by $-1$ on
$B_{\delta/2} \rH^2(X,\bZ)$. 

 Denote by $G\subset \Aut(\Lambda)$ the subgroup generated by $\rho_X(\gamma_X)$ and 
 the isometries $B_{\lambda}=\rho_X(-\otimes \cL)$ with $\cL$ a line bundle of class $\lambda \in \NS(X)$. Clearly $G$ is contained in the image of $\rho_X$. Note that $G$ acts on the lattice
 \[
	\Lambda_{\mathrm{alg}} := B_{\delta/2} \big( \bZ \alpha \oplus \NS(X) \oplus \bZ \beta\big)
\]
and that by our assumption $\NS(X)$ contains a hyperbolic plane.

Let $g\in \Aut^+(\Lambda)$. By Lemma~\ref{lemma:transitivity} applied to $L=\NS(X)$, there exists a $g'\in G$ such that $g'g$ fixes $B_{\delta/2} \delta$.  But then $g'g$ acts on
\[
	(B_{\delta/2} \delta)^\perp = B_{\delta/2} \big( \bZ\alpha \oplus \rH^2(S,\bZ) \oplus \bZ\beta \big)
\]
with determinant $1$, and preserving the Hodge structure and the orientation of a maximal positive subspace. In particular, $g'g$ lies in the image of $\Aut(\cD S)$, and we conclude that $g$ lies in $\im \rho_X$.  
\end{proof}

\end{document}